\documentclass[11pt]{amsart}

\usepackage{amsmath, amsthm, amssymb}

\usepackage[T1]{fontenc}            
\usepackage{palatino}               
\usepackage{comment}

\usepackage[abs]{overpic}		  

\usepackage{xcolor}
\definecolor{indigo}{rgb}{0.29, 0.0, 0.51}  
\usepackage[colorlinks, urlcolor=indigo, linkcolor=indigo, citecolor=indigo]{hyperref}

\usepackage[hcentering, top = 1.5in, total={6in, 8.3in}]{geometry}  

\theoremstyle{plain}
\newtheorem{theorem}{Theorem}
\newtheorem{corollary}[theorem]{Corollary}
\newtheorem{proposition}[theorem]{Proposition}
\newtheorem{lemma}[theorem]{Lemma}

\theoremstyle{definition}

\theoremstyle{remark}
\newtheorem{remark}[theorem]{Remark}

\numberwithin{theorem}{section}

\newcommand{\dfn}[1]{{\em #1}}        
\newcommand{\R}{\mathbb{R}}           
\newcommand{\Q}{\mathbb{Q}}           
\newcommand{\Z}{\mathbb{Z}}           
\newcommand{\N}{\mathbb{N}}           
\DeclareMathOperator{\bd}{\partial}   
\newcommand{\modp}[1]{\,(\!\!\!\!\!\!\mod #1)}

\makeatletter
\newcommand*\bigcdot{\mathpalette\bigcdot@{0.6}}
\newcommand*\bigcdot@[2]{\mathbin{\vcenter{\hbox{\scalebox{#2}{$\m@th#1\bullet$}}}}}
\makeatother

\DeclareMathOperator{\Tight}{Tight}
\DeclareMathOperator\tb{tb}                   
\DeclareMathOperator\rot{rot}                 
\DeclareMathOperator\lk{lk}  			
\DeclareMathOperator\self{sl}                 

\begin{document}

\title{Existence and construction of non-loose knots}

\author{Rima Chatterjee}

\author{John B. Etnyre}

\author{Hyunki Min}

\author{Anubhav Mukherjee}

\address{School of Mathematics \\ Georgia Institute of Technology \\  Atlanta, GA}
\email{etnyre@math.gatech.edu}

\address{Department of Mathematics \\ The University of California, Los Angeles, CA}
\email{hkmin27@math.ucla.edu}

\address{Department of Mathematics \\ Princeton University, Princeton, NJ}
\email{anubhavmaths@princeton.edu}

\address{Department of Mathematics \\ University of Cologne, Germany}
\email{rchatt@math.uni-koeln.de}


\begin{abstract}
In this paper we give necessary and sufficient conditions for a knot type to admit non-loose Legendrian and transverse representatives in some overtwisted contact structure, classify all non-loose rational unknots in lens spaces, and discuss conditions under which non-looseness is preserved under cabling. 
\end{abstract}

\maketitle

\section{Introduction}

Over the last decade, there has been a great deal of study of non-loose Legendrian and transverse knots\footnote{These are Legendrian and transverse knots in overtwisted contact manifolds for which the contact structure restricted to their complement is tight.} in overtwisted contact $3$-manifolds. Most of these results have been either a partial or complete classification of Legendrian knots in certain knot types. For example, \cite{EliashbergFraser98, EtnyreMinMukherjee22pre} classified non-loose Legendrian and transverse unknots and torus knots, respectively, and for various ranges of classical invariants \cite{GeigesOnaran20a, Matkovic20Pre} had previously classified non-loose torus knots. In \cite{GeigesOnaran20b}, non-loose Hopf links were classified, and in \cite{GeigesOnaran15} rational unknots were classified in some lens spaces. In \cite{Ghiggini06b, Ghiggini06, LiscaOzsvathStipsiczSzab09}, various constructions were given and general properties of non-loose knots were studied in \cite{Etnyre13}.

This paper studies the general properties and constructions of non-loose knots. In particular, we give necessary and sufficient conditions for a knot type to admit non-loose representatives, which corrects an error in the second author's paper \cite{Etnyre13}. To achieve this, we should understand certain facts about non-loose representatives of rational unknots, the knot types of cores of Heegaard tori in lens spaces. In proving these facts, we completely classify non-loose rational unknots in any overtwisted lens space, extending the work in \cite{GeigesOnaran15}. We then study the cabling operation and show when ``standard cables" (see Section~\ref{sec:construction}) of non-loose knots in overtwisted manifolds are non-loose. 

\subsection{Existence of non-loose knots}\label{sec:existence}

The most basic question about non-loose Legendrian and transverse knots is when a smooth knot type admits non-loose representatives. We have a complete answer to that question.  Before stating the result we recall that a rational unknot $K$ in $M$ is a knot that has $D^2$ as a rational Seifert surface, see \cite{BakerEtnyre12}. In particular, when $M$ is irreducible, then $K$ is the core of a Heegaard solid torus in a lens space (or $S^3$).  We also note that given any knot $K$ in a $3$-manifold $M$, there is a decomposition of $M$ into $M'\#M''$ so that $K\subset M''$ and $M''\setminus K$ is irreducible. 

\begin{theorem}\label{characterize}
  Let $M$ be an oriented $3$--manifold and $K$ a knot in $M$. Let $M = M' \# M''$ where $K \subset M''$ and $M'' \setminus K$ is irreducible. Then $K$ admits a non-loose Legendrian representative in some overtwisted contact structure on $M$ if and only if 

  \begin{itemize}
    \item $K$ does not intersect an essential sphere transversely once, and
    \item $M'$ admits a tight contact structure.
  \end{itemize}

  Moreover, if $K$ is not the unknot in $M$ and admits a non-loose Legendrian representative, then $K$ admits non-loose Legendrian representatives in at least two overtwisted contact structures. 

  Also, $K$ admits a non-loose transverse representative in some overtwisted contact structure on $M$ if and only if 
  \begin{itemize}
    \item $K$ does not intersect an essential sphere transversely once, 
    \item $K$ is not a rational unknot, and   
    \item $M'$ admits a tight contact structure.
  \end{itemize}

  Moreover, if $K$ admits a non-loose transverse representative, then $K$ admits non-loose transverse representatives in at least two overtwisted contact structures. 
\end{theorem}

\begin{remark}
  In Theorem~\ref{characterize} and Corollary~\ref{characterize-irreducible}, rational unknots include the unknot in $M$.
\end{remark}

For irreducible $3$-manifolds, we have a simpler statement that is an immediate corollary of the previous theorem.

\begin{corollary}\label{characterize-irreducible}
Let $M$ be an irreducible oriented $3$--manifold and $K$ a knot in $M$. Then $K$ admits a non-loose Legendrian representative in some overtwisted contact structure on $M$ if and only if 

\begin{itemize}
  \item $K$ is not contained in a $3$--ball, or
  \item $M$ admits a tight contact structure. 
\end{itemize}  

Moreover, if $K$ is not the unknot in $M$ and admits a non-loose Legendrian representative, then $K$ admits non-loose Legendrian representatives in at least two overtwisted contact structures. 

Also, $K$ admits a non-loose transverse representative in some overtwisted contact structure on $M$ if and only if 

\begin{itemize}
  \item $K$ is not contained in a $3$--ball or $M$ admits a tight contact structure, and   
  \item $K$ is not a rational unknot.
\end{itemize}

Moreover, if $K$ admits a non-loose transverse representative, then $K$ admits non-loose transverse representatives in at least two overtwisted contact structures. \hfill\qed
\end{corollary}

In \cite{Etnyre13}, the second author claimed that all knots in any irreducible $3$-manifold admitted non-loose Legendrian and transverse representatives. This is not the case and the proof in \cite{Etnyre13} will be corrected in the proof of Theorem~\ref{characterize}.

\begin{remark}
  Notice that all knot types in $S^3$ admit non-loose Legendrian representatives and all knot types except the unknot in $S^3$ admit non-loose transverse representatives.
\end{remark}

\begin{remark}
  According to Theorem~\ref{characterize}, the only knot type in a $3$-manifold $M$ (possibly) admitting non-loose Legendrian representatives in only one overtwisted contact structure is the unknot. Indeed, due to the classification of Eliashberg and Fraser \cite{EliashbergFraser09}, we know that the unknot in $S^3$ can be represented by a non-loose Legendrian knot in only one contact structure on $S^3$, namely $\xi_1$.\footnote{Recall there are an integers worth of overtwisted contact structures on $S^3$ denoted by $\xi_n$ where $\xi_0$ is in the same homotopy class of plane field as the tight contact structure.}
\end{remark}

Given what is known about the existence of tight contact structures on $3$--manifolds ({see for example, }\cite{EliashbergThurston96}), the above theorem shows that if none of the summands of $M$ is a rational homology sphere, then any knot type admits non-loose representatives in some overtwisted contact structure. We can also completely determine when knots in Seifert fibered spaces admit non-loose representatives. 
\begin{corollary}\label{maincor}
Let $M_n$ be the result of $(2n-1)$--surgery on the $(2,2n+1)$--torus knot for $n>0$. Any knot in a Seifert fibered spaces admit non-loose Legendrian representatives unless the knot is contained in a ball in $M_n$ or is $S^1\times \{pt\}$ in $S^1\times S^2$.  
\end{corollary}

So the only irreducible $3$-manifolds in which we do not know if a knot can admit a non-loose representative are hyperbolic homology spheres. 

We note that when a knot $K$ in an irreducible manifold $M$ admits a non-loose Legendrian representative in some overtwisted contact structure on $M$, we can say quite a bit more about some of the non-loose representatives. To state our result, we recall some notations. Given a contact structure $\xi$ on a $3$--manifold $M$ and a knot type $K$, we denote by $\mathcal{L}(K)$ the set of all Legendrian knots realizing $K$ up to co-orientation preserving contactomorphism, smoothly isotopic to the identity. If $K$ is null-homologous and oriented, its classical invariants are defined. In this case, let $\mathcal{L}_{(r,t)}(K)$ be the set of the Legendrian representatives with the rotation number equal to $r$ and the Thurston-Bennequin invariant equal to $t$. 

\begin{theorem}\label{allhaveinfinity}
Let $K$ be a knot in an oriented $3$--manifold $M$ that admits non-loose  Legendrian representatives. Assume $K$ is not a rational unknot. Then there are at least two overtwisted contact structures $\xi$ and $\xi'$ on $M$ in which $K$ admits infinitely many non-loose Legendrian representatives with any contact framing. 

In particular, if $K$ is a null-homologous and oriented knot, then in one of the contact structures the set $\mathcal{L}_{(\chi(K)+n,n)} (K)$ contains the non-loose Legendrian representatives $\{L_-^{n,i}\}_{i=0}^\infty$ for $n \in \mathbb{Z}$ and if $-K$ is smoothly isotopic to $K$ (where $-K$ is $K$ with the reversed orientation), then $\mathcal{L}_{(-\chi(K)-n,n)} (K)$ contains $\{L_+^{n,i}\}_{i=0}^\infty$, where 

\begin{enumerate}
  \item $L_\pm^{n,i}$ is obtained from $L_\pm^{n,i-1}$ by adding a convex Giroux torsion layer along a torus parallel to the boundary of a standard neighborhood of $L^{n,i-1}_\pm$.
  \item  $S_{\pm}(L_\pm^{n,i})$ is $L_\pm^{n-1,i}$ and $S_\mp(L_\pm^{n,i})$ is loose. 
\end{enumerate}

In the other contact structures, $\mathcal{L}_{(-\chi(K)+n,n)} (K)$ contains the non-loose Legendrian representatives $\{\overline L_-^{n,i}\}_{i=0}^\infty$ for $n \in \mathbb{Z}$ and if $-K$ is smoothly isotopic to $K$, then $\mathcal{L}_{(\chi(K)-n,n)}(K)$ contains $\{\overline L_+^{n,i}\}_{i=0}^\infty$, where 
\begin{enumerate}
\item $\overline L_\pm^{n,i}$  is obtained from $\overline L_\pm^{n,i-1}$ by adding a convex Giroux torsion layer along a torus parallel to the boundary of a standard neighborhood of $\overline L^{n,i-1}_\pm$.
\item $S_\pm(\overline L_\pm^{n,i})$ is $\overline L_\pm^{n-1,i}$  and $S_\mp(\overline L_\pm^{n,i})$ is loose. 
\end{enumerate}
\end{theorem}
\begin{remark}
The contact structures $\xi$ and $\xi'$ in the theorem are determined by $M$ and $K$, but there does not seem to be an easy way of specifying them explicitly.
\end{remark}
\begin{remark}
This theorem was established for fibered knots by the second author in \cite{Etnyre13}.
\end{remark}

In \cite{Honda00a}, the second author and Honda showed that the classification of transverse knots is the same as the classification of Legendrian knots up to negative stabilization. From this, we have results similar to Theorem~\ref{allhaveinfinity} for the existence of non-loose transverse knots. 

\subsection{Rational unknots}
To prove Theorem~\ref{characterize} we need to deal with the case of rational unknots separately. To do so, we will classify non-loose rational unknots. In \cite{GeigesOnaran15}, Geiges and Onaran did this for rational unknots in $L(p,1)$ and $L(5,2)$. Here we will give the classification for all rational unknots in any lens space. We begin by restating the classification results of Geiges and Onaran and make it clear how the non-loose knots are related by stabilization. 

We first recall that $L(p,q)$ is $-p/q$-surgery on the unknot in $S^3$. We also recall the smooth classification of rational unknots. Given a lens space $L(p,q)$ one can consider the cores $K_0$ and $K_1$ of the Heegaard tori. It is well-known, see \cite{BakerEtnyre12}, that rational unknots up to isotopy are given by 
\[
\{\text{rational unknots in } L(p,q)\}=\begin{cases}
\{K_0\} & p=2\\
\{K_0,-K_0\}& p\not=2, \, q\equiv \pm1 \modp p\\
\{K_0,-K_0,K_1,-K_1\} & \text{otherwise.}
\end{cases}
\]
See Figure~\ref{fig:unknotsSmooth} for surgery descriptions of $K_0$ and $K_1$. It is clear that $K_1$ represents the core of one of the Heegaard tori for $L(p,q)$. We note that $L(p,\overline{q})$ is diffeomorphic to $L(p,q)$ via a diffeomorphism that exchanges the Heegaard tori, and so $K_0$ is clearly the other rational unknot in $L(p,q)$.

\begin{figure}[htbp]
  \vspace{0.1cm}
  \begin{overpic}
  {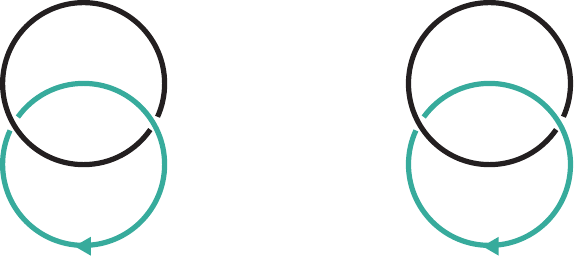}
    \put(63, 124){$-p / \overline{q}$}
    \put(85, 23){$K_0$}
    \put(256, 124){$-p/q$}
    \put(281, 24){$K_1$}
  \end{overpic}
  \caption{The rational unknots in $L(p,q)$. Here, $\overline{q}$ is the positive integer satisfying $1 \leq \overline{q} \leq p$ and $q\overline{q} \equiv 1 \pmod p$.}
  \label{fig:unknotsSmooth}
\end{figure}
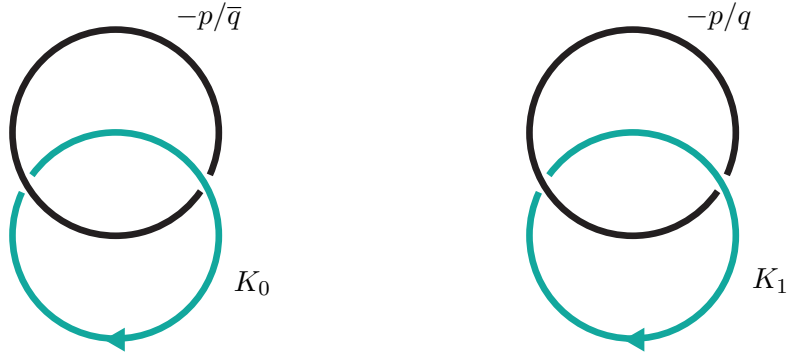

We begin by recalling a convenient means to describe the classification of Legendrian knots. Given a knot type ${K}$ in an overtwisted contact manifold $(M,\xi)$, we denote by $\mathcal{L}({K})$ the coarse equivalence classes of non-loose Legendrian realizations of $K$ in $(M,\xi)$. Consider the map $\Phi:\mathcal{L}(K)\to \Z^2$ that send a Legendrian $L$ to $\Phi(L)=(\rot(L), \tb(L))$. The image of $\Phi$, together with notation about the number of elements sent to a single point, is called the \dfn{mountain range} of $K$. If $K$ is only rationally null-homologous then we have a rational Thurston-Bennequin invariant $\tb_\Q(L)$  and a rational rotation number $\rot_\Q(L)$, see Section~\ref{rational}, and we can consider the similar map $\Phi:\mathcal{L}(K)\to \Q^2$ sending $L$ to $\Phi(L)=(\rot_\Q(L),\tb_\Q(L))$.

We will say a mountain range for a knot type \dfn{contains a $\normalfont\textsf{V}$ based at $(a,b)$} if the image of $\Phi$ contains Legendrian knots $L^0, L^i_\pm$ for $i\in \N$ such that
\begin{align*}
  \tb(L_\pm^i)=b+i &\text{ and } \rot(L_\pm^i)=a\pm(i-1),\\
  \tb(L^0) = b &\text{ and } \rot(L^0) = a,
\end{align*}
and satisfy
\begin{align*}
  &S_\pm(L_\pm^{i+1})=L_\pm^i \text{ and } S_\pm(L_\pm^1)=L^0,\\ 
  &S_\mp(L_\pm^i) \text{ and } S_\pm(L^0) \text{ are loose.}
\end{align*}
See the first drawing of Figure~\ref{fig:mrforru}.  
\begin{figure}[htbp]\small
  \vspace{0.1cm}
  \begin{overpic}
  {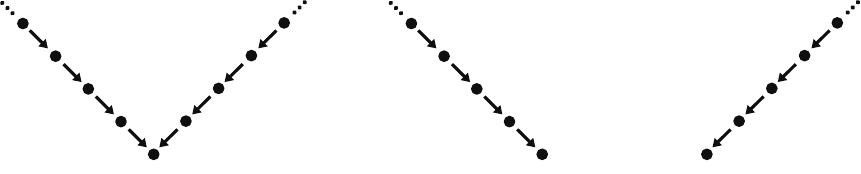}
    \put(63, -4){$(a,b)$}
    \put(251, -4){$(a,b)$}
    \put(329, -4){$(a,b)$}
  \end{overpic}
  \caption{On the left is the mountain range described by a $\normalfont\textsf{V}$ based at $(a,b)$. In the middle is a back slash based at $(a,b)$ and on the right is a forward slash based at $(a,b)$. }
  \label{fig:mrforru}
\end{figure}
We say the mountain range for a knot type \dfn{contains a back slash based at $(a,b)$} if the image of $\Phi$ contains Legendrian knots $L^i$ for $i\geq 0$ and integer such that
\[
\tb(L^i)= b+i \text{ and } \rot(L_i)=a-i
\]
and satisfy 
\begin{align*}
  &S_+(L^{i+1})=L^i,\\ 
  &S_-(L^i) \text{ and } S_+(L^0) \text{ are loose.}
\end{align*}
See the middle drawing of Figure~\ref{fig:mrforru}.

Similarly we say the mountain range for a knot type \dfn{contains a forward slash based at $(a,b)$} if the image of $\Phi$ contains Legendrian knots $L^i$ for $i\geq 0$ and integer such that
\[
\tb(L^i)= b+i \text{ and } \rot(L_i)=a+i
\]
and satisfy 
\begin{align*}
  &S_-(L^{i+1})=L^i,\\ 
  &S_+(L^i) \text{ and } S_-(L^0) \text{ are loose.}
\end{align*}
See the last drawing of Figure~\ref{fig:mrforru}.

We now can restate the result of Geiges and Onaran, together with how Legendrian knots are related via stabilization and which non-loose knots are in the same contact structure (in \cite{GeigesOnaran15}, they showed when the non-loose representatives had the same $d_3$ invariant, but did not compute the (half) Euler class).  Before stating the theorem, recall there are just two rational unknots in $L(p,1)$ and they are $K_0$ and $-K_0$. The theorem is stated for $K_0$ but applies to $-K_0$ by reversing the sign of the rotation numbers. 
\begin{theorem}\label{Lp1}
There are exactly $p$ distinct overtwisted contact structures on $L(p,1)$ that admit non-loose Legendrian representatives of $K_0$. Each contains a $\normalfont\textsf{V}$, one is based at $(0,1/p)$ and the others are based at $\left(1-\frac{2k}p, \frac{p+1}p\right)$, for $k\in \{1,\ldots, p-1\}$. 

The Euler class of the contact structure containing the $\normalfont\textsf{V}$ based at $\left(1-\frac{2k}p, \frac{p+1}p\right)$ is $2k-p$ and $0$ for the $\normalfont\textsf{V}$ based at $(0,1/p)$. 
\end{theorem}

\begin{remark}
When $p$ is odd it is easy to check that the Euler classes of the contact structures on $L(p,1)$ that admit non-loose representatives are all distinct. When $p$ is even one can show that their half-Euler classes, see \cite{Gompf98}, are all distinct, except for two that both have half-Euler class equal to $0$; but those have distinct $d_3$-invariants as computed in \cite{GeigesOnaran15}. We do not discuss the half-Euler classes here, so only make this remark for the readers benefit and to note that there are always at least two contact structures that admit non-loose representatives. 
\end{remark}

We also have the following simple case. Again, the theorem is stated for $K_0$ but applies to $-K_0$ by reversing the sign of the rotation numbers and exchanging the words ``back slash'' and ``forward slash''. 

\begin{theorem}\label{Lpp1}
  For $p\not=2$ there are exactly three overtwisted contact structures on $L(p,p-1)$ in which there are non-loose Legendrian representatives of $K_0$; in one there is $\normalfont\textsf{V}$ based at $\left(0, \frac{2p-1}p\right)$, in another there is a back slash based at $\left(-1+\frac 2p, \frac{p-1}p\right)$ and in the last one there is a forward slash based at $\left(1-\frac 2p, \frac{p-1}p\right)$. The Euler class of the two contact structures containing the slashes is $\pm (2-p)$ while the Euler class of the contact structure containing the $\normalfont\textsf{V}$ is $0$. 
\end{theorem}

\begin{remark}
  Notice that the Euler class of the contact structures containing these three mountain ranges are distinct except for the case when $p=4$. In this case the two slashes have the same Euler class. Their half-Euler classes are distinct and one can compute that their $d_3$-invariants are distinct, but we do not do that here and just note that in all cases we have proven that there are at least two contact structures admitting non-loose rational unknots. 
\end{remark}

We state one more simple case where we must consider both $\pm K_0$ and $\pm K_1$, before discussing the general case. The theorem is stated for $K_0$ and $K_1$ but applies to $-K_0$ and $-K_1$ by reversing the sign of the rotation numbers and exchanging the words ``back slash'' and ``forward slash''.
\begin{theorem}\label{l2n12}
In the lens space $L(2n+1,2)$, for $n\geq 2$, there are $2n+1$ overtwisted contact structures in which $K_0$ has non-loose Legendrian representatives. In one there is a back slash based at $\left(-\frac n{2n+1}, \frac{n+1}{2n+1}\right)$, in another there is a forward slash based at $\left(\frac n{2n+1}, \frac{n+1}{2n+1}\right)$, and in the rest there is a $\normalfont\textsf{V}$ based at $\left(\frac{n-2k}{2n+1}, \frac{n+1}{2n+1}\right)$ for $k=1, \ldots, n-1$ and based at $\left( \frac{n-1-2k}{2n+1},\frac{3n+2}{2n+1}\right)$ for $k=0,\ldots, n-1$. The Euler class of the contact structures containing a slash is $n$ and $n+1$, and the Euler classes of contact structures containing a $\normalfont\textsf{V}$ are all distinct and distinct from $n$ and $n+1$. 

There are $n+2$ contact structures in which $K_1$ has non-loose Legendrian representatives. In one there is a back slash based at $\left(-\frac 1{2n+1}, \frac 2{2n+1}\right)$, in another there is a forward slash based at $\left(\frac 1{2n+1}, \frac 2{2n+1}\right)$, and in the rest there is a $\normalfont\textsf{V}$ based at $\left(\frac{2n-2k}{2n+1}, \frac{2n+3}{2n+1}\right)$ for $k=1,3,\ldots, 2n-1$. Euler classes of all these contact structures are different except for two of the $\normalfont\textsf{V}$s that both have Euler class $0$. 
\end{theorem}

We now give the general classification result. We will denote the negative continued fraction 
\[
a_0-\frac{1}{a_1-\frac{1}{\ldots - \frac{1}{a_n}}}
\]
for $a_i \leq -2$ by $[a_0, a_1, \ldots, a_n]$. 
\begin{theorem}\label{lensgeneral}
Suppose (p,q) is a pair of relatively prime positive integers satisfying $1 < q < p$. Let $-p/q=[a_0,a_1, \ldots, a_n]$ and note that $n\geq 1$ since $q\not= 1$. In $L(p,q)$, the non-loose Legendrian representatives of the rational unknot $K_0$ (and $-K_0$) give 
\[
\begin{cases}
  \left|(a_0+1)\cdots (a_{n-2}+1)(a_{n-1}+2)\right| & n > 1\\
  |a_0 + 2| & n =1
\end{cases}
\]
$\normalfont\textsf{V}$s based at $(r,\overline{q}/p)$, where $1 \leq \overline{q} \leq p$ and $q\overline{q} \equiv 1 \pmod p$, and 
\[
  \begin{cases}
  |(a_0+1)\cdots (a_{n-2}+1)| & n > 1\\
  1 & n = 1
  \end{cases}
\]
 back slashes based at $(r,\overline{q}/p)$ and the same number of forward slashes based at $(r, \overline{q}/p)$ and a formula for $r$ will be given in Section~\ref{sec:ruk}. In addition there are 
 \[
  |(a_0+1)\cdots (a_{n}+1)|
 \] 
 $\normalfont\textsf{V}$s based at $(r,(\overline{q}+p)/p)$. A formula for $r$ and the Euler class of the contact structure containing each of theses mountain ranges will also be given Section~\ref{sec:ruk}, but we note here that the mountain ranges are in at least two distinct contact structures.

  The non-loose Legendrian representatives of the rational unknot $K_1$ (and $-K_1$) give 
  \[
  \left|(a_{1}+2)(a_{2}+1)\cdots (a_n+1)\right| 
\]
$\normalfont\textsf{V}$s based at $(r,q/p)$, and 
\[
  \begin{cases}
  |(a_2+1)\cdots (a_n+1)| & n > 1\\
  1 & n = 1
  \end{cases}
\]
 back slashes based at $(r,q/p)$ and the same number of forward slashes based at $(r, q/p)$ and a formula for $r$ will be given in Section~\ref{sec:ruk}. In addition there are 
 \[
  |(a_0+1)\cdots (a_{n}+1)|
 \] 
 $\normalfont\textsf{V}$s based at $(r,(q+p)/p)$. A formula for $r$ and the Euler class of the contact structure containing each of theses mountain ranges will also be given Section~\ref{sec:ruk}, but we note here that the mountain ranges are in at least two distinct contact structures. 
\end{theorem}

See Firgure~\ref{fig:unknotsLegendrian1} and~\ref{fig:unknotsLegendrian2} for the diagrams of non-loose Legendrian representatives of rational unknots. We will show that this are indeed the desired diagrams in Section~\ref{sdru}.

\begin{figure}[htbp]\small
  \vspace{0.1cm}
  \begin{overpic}
  {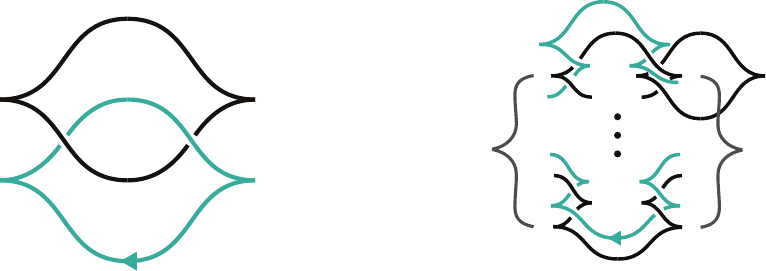}
    \put(90, 120){$\left(\frac{q}{p+q}\right)$}
    \put(96, 13){$K_1$}
    \put(188, 57){$2-a_0-k$}
    \put(366, 57){$k$}
    \put(321, 4){$(+1)$}
    \put(361, 110){$(s+1)$}
  \end{overpic}
  \caption{Left: non-loose Legendrian representatives of $K_1$ with $\tb_{\Q} = q/p$. Right: non-loose Legendrian representatives of $K_1$ with $\tb_{\Q} = (p+q)/p$. Here, $1 \leq k \leq 1-a_0$ and $s = [a_1,\ldots,a_n]$ where $-p/q = [a_0,\ldots,a_n]$. For $K_0$, replace $q$ with $\overline{q}$ where $1 \leq \overline{q} \leq p$ and $q\overline{q} \equiv 1 \pmod p$.}
  \label{fig:unknotsLegendrian1}
\end{figure}

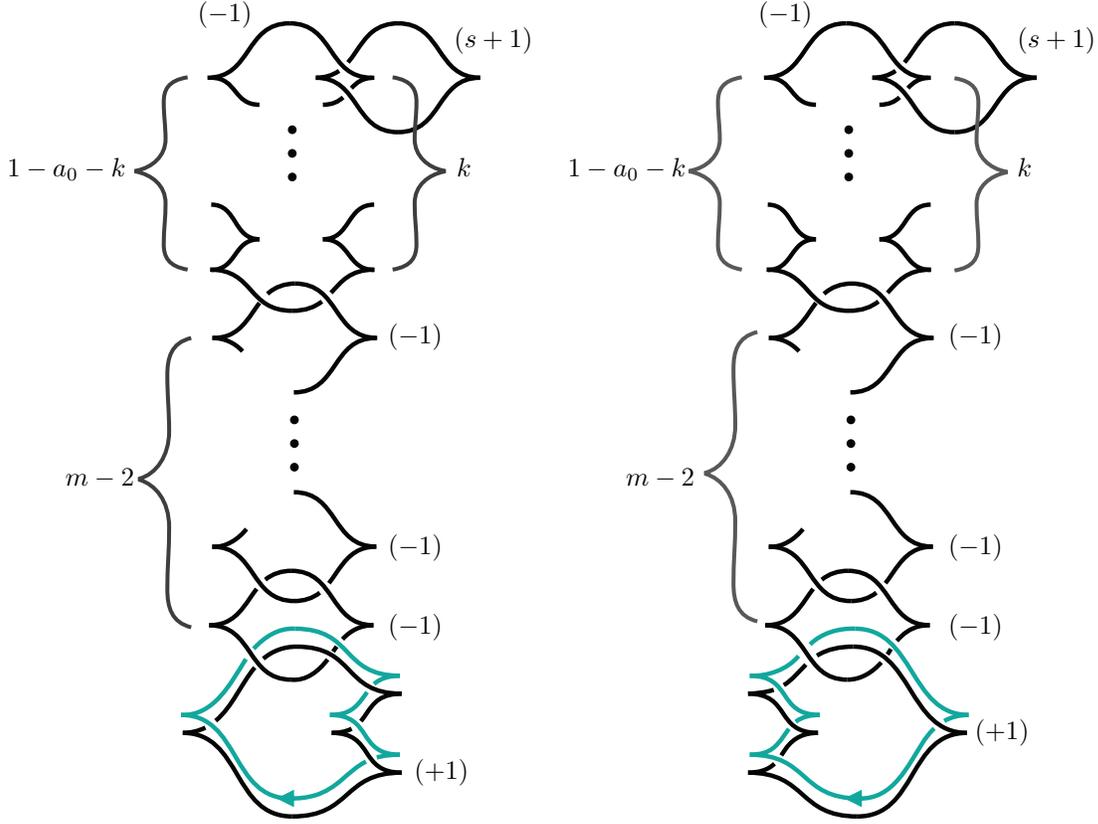
\begin{figure}[htbp]\small
  \vspace{0.1cm}
  \begin{overpic}
  {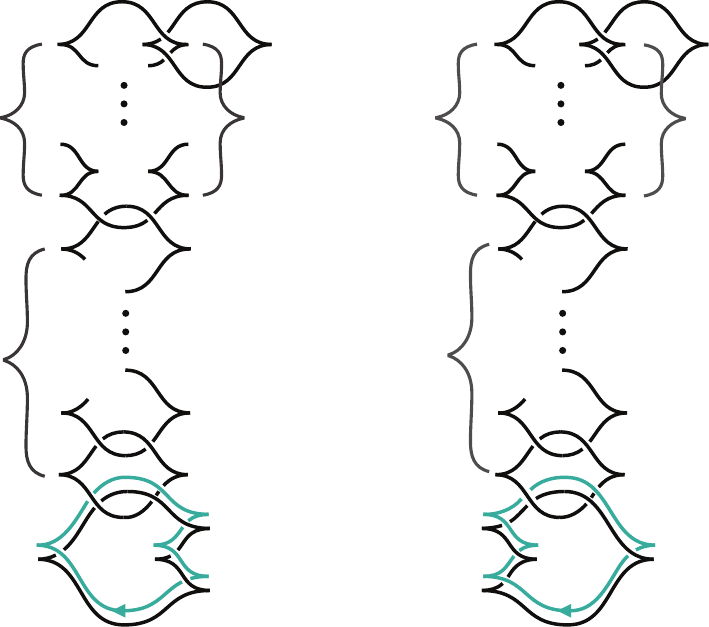}
    \put(-48, 243){$1-a_0-k$}
    \put(122, 243){$k$}
    \put(24, 302){$(-1)$}
    \put(121, 292){$(s+1)$}
    \put(-26, 126){$m-2$}
    \put(96, 180){$(-1)$}
    \put(96, 100){$(-1)$}
    \put(96, 70){$(-1)$}
    \put(106, 15){$(+1)$}

    \put(164, 243){$1-a_0-k$}
    \put(334, 243){$k$}
    \put(236, 302){$(-1)$}
    \put(334, 292){$(s+1)$}
    \put(186, 126){$m-2$}
    \put(308, 180){$(-1)$}
    \put(308, 100){$(-1)$}
    \put(308, 70){$(-1)$}
    \put(318, 30){$(+1)$}
  \end{overpic}
  \caption{Non-loose Legendrian representatives of $K_1$ with $\tb_{\Q} = q/p + m$ for $m \geq 2$. Here, $1 \leq k \leq -a_0$ and $s = [a_1,\ldots,a_n]$ where $-p/q = [a_0,\ldots,a_n]$. For $K_0$, replace $q$ with $\overline{q}$ where $1 \leq \overline{q} \leq p$ and $q\overline{q} \equiv 1 \pmod p$.}
  \label{fig:unknotsLegendrian2}
\end{figure}

From the above results we have the following immediate corollary that we need for the proof of Theorem~\ref{characterize}. 
\begin{corollary}\label{courseRU}
Every rational unknot in a lens space, except for the unknot in $S^3$, admits non-loose Legendrian representatives in at least two contact structures. Also, every rational unknot does not admit non-loose transverse representatives. 
\end{corollary}

\subsection{Construction of non-loose knots via cabling} \label{sec:construction}
We begin by defining ``standard cables" of a null-homologous Legendrian knot and then exploring how non-looseness of a knot and its standard cable are related. We first recall that our slope convention is $\frac{\text{meridian}}{\text{longitude}}$. Given a null-homologous Legendrian knot $L$ in any contact manifold, let $N$ be a standard neighborhood of $L$ with ruling slope $q/p$. If $q/p > \tb(L)$, then we define the \dfn{standard positive $(p,q)$-cable} of $L$ as a ruling curve on $\partial N$ and denote by $L_{p,q}$.  
\begin{theorem}\label{thm:poscable}
  Let $(M,\xi)$ be an overtwisted contact $3$--manifold and $K$ a null-homologous knot in $M$. Suppose $L$ is a Legendrian representative of $K$ in $(M,\xi)$. Then for $q/p > \tb(L)$, the standard $(p,q)$-cable $L_{p,q}$ of $L$ is non-loose if and only if $L$ is non-loose.
  
\end{theorem}

Now suppose that $L$ is a Legendrian knot and $q/p \leq \tb(L)$. If $N$ is a standard neighborhood of $L$, then inside of $N$ we can find a convex torus $T$ parallel to $\bd N$ with two dividing curves of slope $q/p$. We can define the \dfn{standard negative $(p,q)$-cable of $L$} as a Legendrian divide of $T$, but we note that there is some ambiguity here. If $q/p\in[\tb(L)-n, \tb(L)-n+1)$ then there are actually $n+1$ distinct tori with dividing slope $q/p$. See Section~\ref{kot} for more details, but in brief there are $n+1$ convex tori inside $N$ with dividing slope $\tb(L)-n$ and they come from stabilizing $L$, $n$ times (there are $n+1$ ways to do this), and then the tori $T$ with dividing slope $q/p$ are determined by those stabilizations. Thus we see that a standard cable of $L$ is also a standard cable of a $n-1$ times stabilization $L'$ of $L$ and then $q/p\in(\tb(L')-1,\tb(L'))$. Thus we will only define the standard cable of a Legendrian knot $L$ if $q/p$ satisfies this condition. Then there are only two possibilities for the torus $T$ and we denote the corresponding dividing curves by $L_{p,q}^\pm$ depending on whether the torus $T$ inside $N$ but outside a positive or negative stabilization of $L$. We call $L_{p,q}^\pm$ the \dfn{$\pm$ standard negative $(p,q)$-cable of $L$}. 
\begin{theorem}\label{thm:negcable} 
  Let $(M,\xi)$ be an overtwisted contact $3$--manifold and $L$ be a Legendrian knot in $(M,\xi)$. 
  If  $S_\pm(L)$ is non-loose, then 
  \[
 \text{$L_{p,q}^\pm$ is non-loose for all $q/p\in (\tb(L)-1,\tb(L))$.}
  \]
  Here $S_\pm(L)$ denotes the $\pm$ stabilization of $L$.
  \end{theorem}

We notice that whenever we construct a $\pm$ standard negative $(p,q)$-cable one can also construct a standard positive $(p,q)$-cable. Specifically for a Legendrian knot $L$ we define the $L^\pm_{p,q}$ , standard $\pm$ negative cable of $L$ as above, and we can also define $(S_\pm(L))_{p,q}$ the standard positive $(p,q)$-cable of $S_\pm(L)$, as above. We note the following simple result.
\begin{proposition}\label{relation}
With the notation above $(S_\pm(L))_{p,q}$ is the $n$ fold $\pm$ stabilization of $L^\pm_{p,q}$, where $n=\tb(S_\pm(L))\cdot \frac qp$.
\end{proposition}
Our two theorems above show that if $S_\pm(L)$ is non-loose then not only is $L^\pm_{p,q}$ non-loose but so is its $n$ fold $\pm$ stabilization. This proposition shows that we could define the standard positive $(p,q)$-cable in terms of the standard negative cable in certain circumstances, but the standard positive $(p,q)$-cable is defined in greater generality (in particular, we don't need to know that $L$ can be destabilized). 

\begin{figure}[htbp]\small
  \vspace{0.1cm}
  \begin{overpic}
  {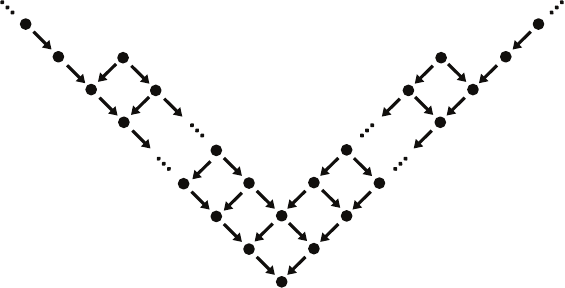}
    \put(-12, 108){$4n+2$}
    \put(99, -1){$2n+1$}
  \end{overpic}
  \caption{Mountain range for the $(2,2n+1)$-torus knots in $\xi_1$. }
  \label{fig:mrtorus}
\end{figure}
\begin{remark}
We notice that while our ``standard cables" allow us to construct many non-loose knots in cabled knot types, there are many non-loose representatives that do not come from this construction. We demonstrate this with the $(2,2n+1)$-torus knots which are of course $(2,2n+1)$-cables of the unknot. In Figure~\ref{fig:mrtorus} the mountain range for the non-loose $(2,2n+1)$-torus knots in $(S^3,\xi_1)$ is given. Consider the non-loose unknot with $\tb=1$ in $(S^3,\xi_1)$. Since $(2n+1)/2>1$ we see that the $(2,2n+1)$-cable is a positive cable of this Legendrian knot, and the standard positive cable will have $\tb=2n+3$. This is the middle dot in the third row from the bottom of the figure, we call this the vertex of the inner $V$. The upper two end points of the inner $V$ are the standard negative cables of the two Legendrian unknots with $tb=n+1$ and all the other points in the inner $V$ correspond to positive cables of unknots with $\tb<n+1$. These are all the non-loose representatives of the $(2,2n+1)$-cable of the unknot that can be seen via our construction. The infinitely many representatives in the outer $\normalfont\textsf{V}$ are not ``standard cables". Similarly, there are two other overtwisted contact structures, $\xi_0$ and $\xi_{1-2n}$, that admit non-loose representatives of the $(2,2n+1)$-torus knot and none of them can be seen by ``standard cables" either. 

We also note, that no non-loose negative torus knots are seen as standard cables of non-loose unknots. 
\end{remark}

Given a transverse knot $K$ we know there is a sequence of Legendrian approximations $L_n$ for $n$ sufficiently small such that, $\tb(L_n)=n$,  $S_-(L_n)=L_{n-1}$, and the transverse push-off of $L_n$ is $K$, see \cite[Section~2.2]{EtnyreHonda01}. Give any rational number $q/p$ then for any $n< \lfloor q/p\rfloor$ we can consider the standard $(p,q)$-cable, $(L_n)_{p,q}$, of $L_n$. We have the following simple observation
\begin{equation}
S^p_-((L_{n})_{p,q})=(L_{n+1})_{p,q},
\end{equation}
which follows from Lemma~\ref{seestab}. Since the transverse push-off of a Legendrian knot is transversely isotopic to the transverse push-off of the Legendrian knot after any number of negative stabilizations, we define the \dfn{standard $(p,q)$-cable of a transverse knot $K$}, denoted $K_{p,q}$, to be the transverse push of of $(L_n)_{p,q}$ for any Legendrian approximation $L_n$ with $n$ less than $q/p$. We can now state the result about cabling non-loose transverse knots.

\begin{theorem}\label{transversecable}
Let $(M,\xi)$ be an overtwisted contact $3$--manifold and $K$ be a transverse knot in $(M,\xi)$. Then $K_{p,q}$ is non-loose for any $(p,q)$ if and only if $K$ is non-loose. 
\end{theorem}

It is well-known that there are non-loose transversely non-simple knots, see \cite{Etnyre13} and Theoerem~\ref{allhaveinfinity} above, but all of the currently known examples rely on different amounts of Giroux torsion in the knot complement. Using cables we give the first examples of non-simple non-loose transverse knots with Giroux torsion zero in their complements. 

\begin{theorem}\label{thm:transnonsimple}
In the overtwisted contact structure $\xi_2$ on $S^3$, the $(2n+1,2)$-cable of the left-handed trefoil has at least $n$ distinct non-loose Legendrian representatives with $\tb = 4n+2$, $\rot = 2n-1$ and no Giroux torsion in its complement. There are also at least $n$ distinct non-loose transverse representatives with $\self=2n+3$ and no Giroux torsion in its complement. 
\end{theorem}

\begin{remark}
We have similar results for all $(p,q)$-cables of the left handed trefoil with $q/p\in (0,1)$, but we only prove the above theorem since it already shows the existence of arbitrarily many transverse knots with the same self-linking and Giroux torsion zero. 
\end{remark}

\noindent
{\bf Acknowledgments.} The first author is partially supported by SFB/TRR 191 ``Symplectic Structures in Geometry, Algebra and Dynamics, funded by the Deutsche Forschungsgemeinschaff (Project- ID 281071066-TRR 191)'' and the Georgia Institute of Technology's Elaine M. Hubbard Distinguished Faculty Award. The second author was partially supported by National Science Foundation grants DMS-1906414 and DMS-2203312 and the Georgia Institute of Technology's Elaine M. Hubbard Distinguished Faculty Award. The authors thank Kenneth Baker for a helpful discussion about Theorem~\ref{characterize}.

\section{Background}\label{background}

We assume the reader is familiar with the basic ideas of contact geometry, Legendrian knots, and convex surfaces, as can be found in, for example, \cite{EtnyreHonda01, Geiges08,Honda00a}. We review parts of this below for the convenience of the reader and to establish notation. In Section~\ref{fgraph} we recall the Farey graph and discuss its relation to curves on tori and in the following section we recall the classification of tight contact structures on $T^2\times[0,1]$, solid tori, and lens spaces. In Section~\ref{kot} we review several aspects of knots in contact manifolds, such as standard neighborhoods and how these relate to stabilization. The last two sections discuss bypasses and rationally null-homologous knots, respectively. 

\subsection{The Farey graph}\label{fgraph}
The Farey graph is an essential tool to keep track of embedded essential curves on a torus. Recall that once a basis for $H_1(T^2)$ is chosen the embedded essential curves on the torus are in one to one correspondence with the rational numbers union infinity.

The Farey graph is constructed in the following way, see Figure~\ref{fareygraph}. Consider the unit disc in the $xy$-plane. Label the point $(0,1)$ as $0=\frac{0}{1}$ and $(0,-1)$ as $\infty=\frac{1}{0}$. Connect these two points by a straight line. Now if a point on the boundary of the disk has a positive $x$-coordinate and it lies half way between two points labeled $\frac{a}{b}$ and $\frac{c}{d}$ then we label it as $\frac{a+c}{b+d}$. We call this as the ``Farey sum" of $\frac{a}{b}$ and $\frac{c}{d}$ and write as $\frac{a}{b}\oplus\frac{c}{d}$ . Now we connect $\frac{a+c}{b+d}$ with the two points by hyperbolic geodesics (we can consider a hyperbolic metric on the interior of the disk). We iterate this process until all positive rational numbers are labeled point on the boundary of the unit disk. Now we do the same for points on the unit circle with negative $x$-coordinate, except that here we consider $\infty$ as $\frac{-1}{0}$.
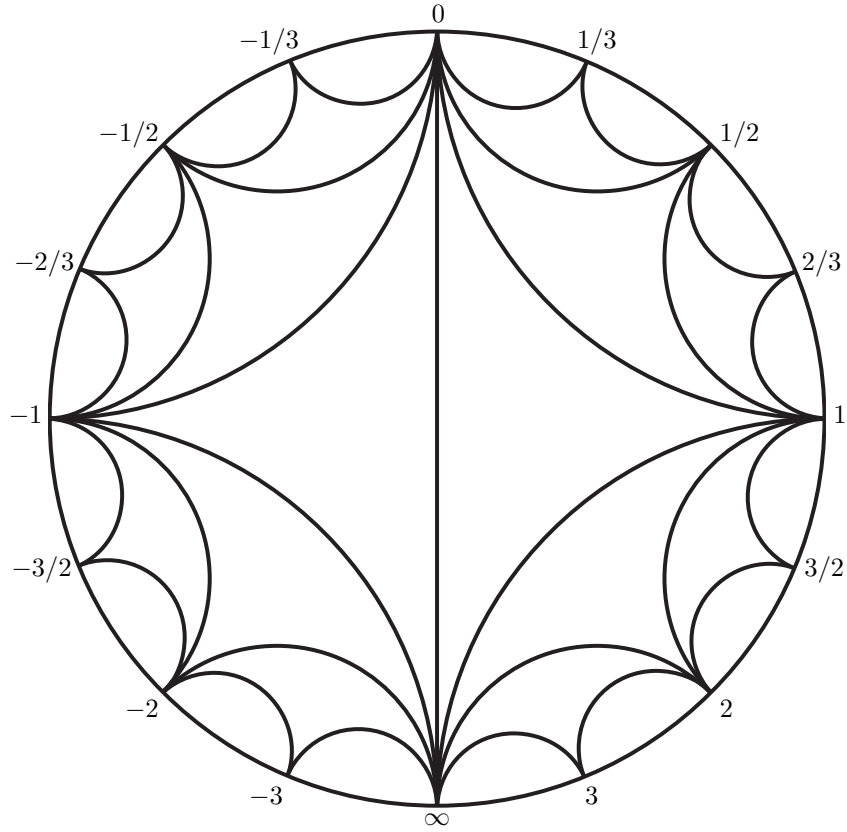
\begin{figure}[htb]{\small
\begin{overpic}
{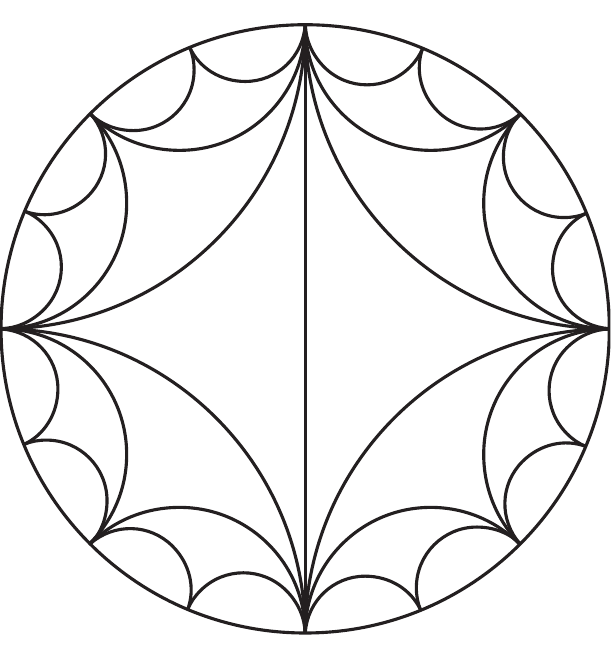}
\put(142, 2){$\infty$}
\put(145, 306){$0$}
\put(-15, 154){$-1$}
\put(297, 154){$1$}
\put(29, 43){$-2$}
\put(254, 43){$2$}
\put(19, 261){$-1/2$}
\put(254, 261){$1/2$}
\put(72, 296){$-1/3$}
\put(200, 296){$1/3$}
\put(-13, 212){$-2/3$}
\put(285, 212){$2/3$}
\put(-14, 96){$-3/2$}
\put(286, 96){$3/2$}
\put(76, 10){$-3$}
\put(203, 10){$3$}
\end{overpic}
\caption{The Farey graph.}
\label{fareygraph}}
\end{figure}

In this paper we will use $\frac{a}{b}\oplus k\frac{c}{d}$ for the iterated Farey sum where we add $\frac{a}{b}$ to $\frac{c}{d}$, $k$-times.

Note that, two embedded curves on the torus with slopes $r$ and $s$ will form a basis of $H_1(T^2)$ if and only if there is an edge between them in the Farey graph.
Here we also introduce the dot product of two rational numbers $\frac{a}{b}\bigcdot \frac{c}{d}=ad-bc$ and note that $|\frac{a}{b}\bigcdot \frac{c}{d}|$  is the minimum number of times the curves with slopes $\frac{a}{b}$ and $\frac{c}{d}$ can intersect.

We have the following well-known lemma, for example see \cite{EtnyreLaFountainTosun12}
\begin{lemma}
	Suppose $q/p<-1$. Given $q/p=[a_1,\dots,a_n]$, let $(q/p)^c=[a_1,\dots,a_n+1]$ and $(q/p)^a=[a_1,\dots,a_{n-1}]$. There will be an edge in the Farey graph between each pair of numbers $q/p, (q/p)^c$ and $(q/p)^a$. Moreover, $(q/p)^c$ will be farthest clockwise point from $q/p$ that is larger that $q/p$ with an edge to $q/p$, while $(q/p)^a$ will be the farthest anti-clockwise point from $q/p$ that is less than $q/p$ with an edge to $q/p$.
\end{lemma}

In the above lemma if $a_n+1=-1$, then we consider $[a_1,\dots,a_n+1]$ to be  $[a_1,\dots,a_{n-1}+1]$. Also if $q/p$ is a negative integer, then $(q/p)^a = \infty$.

A path in the Farey graph is a sequence of elements $p_1,\dots ,p_k$ in $\Q \cup \infty$ moving clockwise such that each $p_i$ is connected to $p_{i+1}$ by an edge in the Farey graph, for $i<k$. Let $P$ be the minimal path in the Farey graph that starts at $s_0$ and goes clockwise to $s_1$. We say $P$ is a {\it decorated path} if all of its edges are decorated by a $+$ or a $-$. We call a path in the Farey graph a {\it continued fraction block} if there is a change of basis such that the path goes from $0$ clockwise to $n$ for some positive $n$. We say two choices of signs on the continued fraction block are related by {\it shuffling} if the number of $+$ signs in the continued fraction blocks are the same. 

Now we introduce some notations that we will frequently use in this paper. Given two numbers in $\Q\cup\infty$ we let $[r,s]$ denotes all those numbers in $\Q\cup\infty$ which are clockwise to $r$ in the Farey graph and anticlockwise to $s$.

\subsection{Contact structures on $T^2\times [0,1]$, solid tori and lens spaces} 

Here we briefly recall the classification of tight contact structures on $T^2\times [0,1]$, $S^1\times D^2$, {and} lens spaces due to Giroux \cite{Giroux2} and Honda\cite{Honda00a}. We discuss the classification along the lines of Honda.
\subsubsection{Contact structures on $T^2\times [0,1]$}
Consider a contact structure $\xi$ on $T^2\times[0,1]$ that has convex boundary with dividing curves of slope $s_0$ on $T_0=T^2\times\{0\}$ and $s_1$ on $T_1=T^2\times\{1\}$. We also assume that the number of dividing curves is two on each boundary component. We say $\xi$ is {\it minimally twisting} if any convex torus in $T^2\times[0,1]$ parallel to $T_0$ has dividing curves with slope in $[s_0,s_1]$. We denote the minimally twisting contact structures, up to isotopy, on $T^2\times[0,1]$ with the above boundary conditions as $\Tight^{min}(T^2\times[0,1]; s_0, s_1)$. Giroux \cite{Giroux2} and Honda \cite{Honda00a} classified contact structures in $\Tight^{min}(T^2\times[0,1]; s_0, s_1)$, establishing the following results. 

\begin{theorem}
Each decorated minimal path in the Farey graph from $s_0$ clockwise to $s_1$ describes an element of $\Tight^{min}(T^2\times[0,1]; s_0, s_1)$. Two such decorated paths will describe the same contact structure if and only if the decorations differs by shuffling in the continued fraction blocks.
\end{theorem}
Notice that if $s_0$ and $s_1$ shares exactly one edge in the Farey graph then there are exactly two tight contact structures in $\Tight^{min}(T^2\times[0,1]; s_0, s_1)$. These are called {\it basic slices} and the correspondence in the theorem can be understood via stacking basic slices according to the decoration in the path that describes the contact structure. The two different contact structures on a basic slice can be distinguished by their relative Euler class we call them {\it positive} and {\it negative basic slices}. 

The relative Euler class of the contact structure in $\Tight^{min}(T^2\times[0,1]; s_0, s_1)$  can be computed as follows: let $s_0=r_0, r_2,\ldots r_k=s_1$ be the vertices of the minimal path from $s_0$ to $s_1$ and let $\epsilon_i$ be the sign of the  basic slice corresponding to $r_{i-1}$ and $r_i$. Then the relative Euler class of the contact structure associated with this path is Poincar\'e dual to the curve 
\[\sum_{i=1}^{n}\epsilon_i(r_i\ominus r_{i-1})\] where $\frac{a}{b}\ominus \frac{c}{d}=\frac{a-c}{b-d}$.

Now suppose $P$ is a non-minimal path is the Farey graph. So there will be a vertex $v$ in $P$ such that there is an edge between its neighboring vertices $v'$ and $v''$. We can shorten this path by removing $v$ and the two edges and replacing it by the edge between $v'$ and $v''$. We call the new path $P'$. If $P$ were a decorated path, then we call the shortening to get $P'$ is \dfn{inconsistent} if the signs of the edges removed are different and \dfn{consistent} otherwise. When the shortening is consistent, then we can decorate the new edge in $P'$ by the sign of the removed edges, thus $P'$ is a new decorated path. 

Given any decorated path in the Farey graph, even non-minimal, one can construct a contact structure on $T^2\times [0,1]$ by stacking basic slices. We are interested to know when these paths will give us tight contact structures.  We have the following result due to Honda \cite{Honda00a}.

\begin{theorem}
Let $\xi$ be a contact structure on $T^2\times [0,1]$ described by a non-minimal decorated path $P$ in the Farey graph from $s_0$ to $s_1$. Then $\xi$ is tight if and only if one may consistently shorten the path to a shortest path from $s_0$ to $s_1$. 	
	\end{theorem}

We now discuss convex Giroux torsion. Consider $\xi=\ker(\sin2\pi z\, dx+\cos2\pi z\,  dy)$ on $T^2\times\R$ where $(x,y)$ is the cooordinate on $T^2$ and $z$ is the coordinate on $\R$. Consider the region $T^2\times[0,k]$ for $k\in\frac{1}{2}\N$ and notice that the contact planes twist $k$ times as $z$ goes from $0$ to $k$. We can perturb $T^2\times \{0\}$ and $T^2\times\{k\}$ so that they become convex with two dividing curves of slope $0$. Let $\xi^k$ denotes the resulting contact structure on $T^2\times[0,1]$ (After identifying $T^2\times [0,k]$ with $T^2\times [0,1]$). 

For $k \in \frac12\N$, we call $(T^2\times[0,1],\xi^k)$ a {\it convex $k$ Giroux torsion layer} and if it embeds into a contact manifold $(M,\xi)$, we say $(M,\xi)$ has {\it convex k Giroux torsion}. We say $(M,\xi)$ has exactly $k$ Giroux torsion if one can embed $(T^2\times[0,1],\xi^k)$ into $(M,\xi)$ but cannot embed $(T^2\times[0,1],\xi^{k+\frac{1}{2}})$ in $(M,\xi)$. On the other hand $(M,\xi)$ has no convex Giroux torsion or zero Giroux torsion if $(T^2\times[0,1],\xi^k)$ does not embed in $(M,\xi)$ for any $k\in\frac{1}{2}\N$.

\subsubsection{Contact structures on solid tori}\label{conttori}
Now we turn to the contact structures on $S^1\times D^2$. While we will usually use ``standard" coordinates on a solid torus so that the meridional slope is $-\infty=\infty$, it will sometimes be convenient to describe solid tori using other coordinate. We set up notation for this now. Consider $T^2\times [0,1]$ and choose a basis for $H_1(T^2)$ so that we may denote curves on $T^2$ by rational numbers (union infinity). Given $r\in\Q\cup\{\infty\}$ we can foliate $T^2\times \{0\}$ by curves of slope $r$. Let $S_r$ be the result collapsing each leaf in the foliation of $T^2\times \{0\}$ to a point. One may easily verify $S_r$ is a solid torus with meridional slope $r$. We say \dfn{$S_r$ is a solid torus with lower meridian $r$}. We could similarly foliate $T^2\times\{1\}$ by curves of slope $r$ and collapse them to obtain $S^r$. This is the \dfn{solid torus with upper meridian $r$}. Notice that the ``standard" solid torus is $S_\infty$, and if we do not indicate otherwise, this is the torus we are talking about. 

Let $\Tight(S_r,s)$ denote the tight contact structures on the solid torus $S_r$ with lower meridian $r$ and convex boundary having two dividing curves of slope $s$, and similarly for $\Tight(S^r,s)$. We call a path $P$ in the Farey graph moving clockwise \dfn{mostly decorated} if all but the first or last edge has a sign. We call it \dfn{mostly upper decorated} if there are signs on all but the first edge and \dfn{mostly lower decorated} if all the edges have a sign except for the last edge. 
Giroux \cite{Giroux2} and Honda \cite{Honda00a} classified tight contact structures on a solid torus. By a change of basis and using the notation above their result is the following. 

\begin{theorem}
	\label{thm:continued_block}
	The elements of $\Tight(S_r,s)$ are in one-to-one correspondence with equivalence classes of mostly upper decorated minimal paths in the Farey graph from $r$ clockwise to $s$ where two such paths are considered equivalent if they differ by shuffling in continued fraction blocks. 

Similarly, elements in $\Tight(S^r,s)$ are in one-to-one correspondence with equivalence classes of mostly lower decorated minimal paths in the Farey graph from $r$ clockwise to $s$ where two such paths are considered equivalent if they differ by shuffling in continued fraction blocks. 
\end{theorem}
Just like for thickened tori, one can describe a contact structure on a solid tori by any path in the Farey graph, even if not minimal, but the contact structure might not be tight. Again, just as for thickened tori, if $P$ is a non-miminal path we can talk about inconsistent shortenings (same definition as for thickened tori) and consistent shortenings, which are any shortenings that are not inconsistent. When considering solid tori, there is a consistent shortening that was not seen in the thickened torus case.  Specifically, if one of the two edges that is removed when shortening the path is the unlabeled one, then we can shorten the path and leave the new edge unlabeled. Notice that this will remove one of the edges that had a sign (that is the unlabeled edge can consistently shorten, and ``absorb", any signed edge). 

\begin{theorem}
Let $\xi$ be a contact structure on $S_r$ described by a non-minimal mostly upper decorated path $P$ in the Farey graph from $r$ to $s$. Then $\xi$ is tight if and only if one may consistently shorten the path to a shortest path from $r$ to $s$. We have a similar statement for $S^r$. 
	\end{theorem}

We will discuss non-minimal paths further in the context of neighborhoods of Legendrian knots in the next section. 

\subsubsection{Contact structures on lens spaces}
The lens space $L(p,q)$ is defined to be $-p/q$ surgery on the unknot in $S^3$. Equivalently, we can think of $L(p,q)$ as $T^2\times [0,1]$ with the curves of slope $-p/q$ collapsed on $T^2\times\{0\}$ and curves of slope $0$ collapsed on $T^2\times\{1\}$, which could further be described as the result of gluing the torus $S_{-p/q}$ with lower meridian $-p/q$ to the torus $S^0$ with upper meridian $0$ by the identity along the boundary. Giroux \cite{Giroux2} and Honda \cite{Honda00a} classified tight contact structure on $L(p,q)$ proving the following. 
\begin{theorem}
Let $P$ be a minimal path in the Farey graph from $-p/q$ clockwise to $0$. The tight contact structures on $L(p,q)$ are in one-to-one correspondence with assignments of signs to all but the first and last edge in $P$ up to shuffling in continued fraction blocks. 
\end{theorem}

\subsection{Knots in contact manifolds}\label{kot}
Inside any open set containing a Legendrian knot $L$ in a contact manifold $(M,\xi)$ there is a \dfn{standard neighborhood} $N(L)$ of $L$. This is a solid torus on which $\xi$ is tight and $\partial N(K)$ is convex with two dividing curves of slope $\tb(L)$. One may arrange, by a small isotopy of $\partial N(L)$ that the characteristic foliation consists of two lines of singularities called \dfn{Legendrian divides} and curves of slope $s$, for any $s\not=\tb(L)$. These latter curves are called \dfn{ruling curves}. (For most of this subsection we will be considering standard coordinates on a solid torus, so when we say solid torus we mean a solid torus with lower meridian $\infty$ in the terminology of Section~\ref{conttori}.) 

One may reverse the above discussion. Given a solid torus $S$ in a contact manifold $(M,\xi)$ on which $\xi$ is tight and with convex boundary having two dividing curves of slope $n\in\Z$, then there is a unique Legendrian knot $L_S$ with $\tb=n$ having $S$ as its standard neighborhood. 

Given a Legendrian knot $L$, we can stabilize $L$ in two ways, $S_\pm(L)$, and choose a standard neighborhood $N(S_\pm(L))$ of $S_\pm(L)$ in $N(L)$. Notice that $N(L)\setminus N(S_\pm(L))$ is a $\pm$ basic slice with boundary slopes $\tb(L)-1$ and $\tb(L)$. It turns out all the sub-tori with convex boundary having integral slopes can be formed this way. A corollary of classification results in Section~\ref{conttori} that will be useful to us is the following result. 
\begin{corollary}\label{subtori}
Let $S$ be a solid torus and $\xi$ a tight contact structure on $S$ such that $\partial S$ is convex with two dividing curves of slope $n\in\Z$. For each positive integer $k$ there are exactly $k+1$ distinct solid tori $S_0,\ldots, S_{k}$ in $S$ with convex boundary having two dividing curves of slope $n-k$.  They are determined by the contact structure on $S\setminus S_i$ which in turn is determined by a signed minimal path in the Farey graph from $n-k$ to $n$. If $k'$ is an integer larger than $k$ then a solid torus $S'_j$ with convex boundary having two dividing curves of slope $n-k'$ is contained in $S_i$ if and only if the signed path associated to $S\setminus S_i$ is, after possibly shuffling the signs, a sub path of the signed path associated to $S\setminus S'_j$.
\end{corollary}
We also notice that any solid torus with convex boundary has a unique maximal Thurston-Bennequin invariant Legendrian knot in its interior. 
\begin{corollary}\label{lemma:rational}
Let $S$ is a solid torus with convex boundary having dividing slope $r\in\mathbb{Q}\setminus\mathbb{Z}$. Then there exists a unique solid torus with convex boundary having two dividing curves of slope $\lfloor{r}\rfloor$ inside $S$ and that torus contains any convex torus with smaller dividing slope.
\end{corollary}

We now recall two well-known results about Legendrian knots sitting on a convex torus. 
\begin{lemma}\label{canstablize}
Suppose $L$ sits on a convex torus $T$ and $L'$ is a ruling curve on $T$ in the same homology class as $L$ (or a copy of $T$ obtained by isotoping $T$ through convex tori). Then $L$ is obtained from $L'$ by some number of stabilizations. 
\end{lemma}

\begin{lemma}\label{isruling}
Suppose $N$ is a standard neighborhood of a Legendrian knot $L$ and $\partial N$ has longitudinal ruling curves of slope larger than the contact framing (that is $\tb(L)$). Then the ruling curves of $\partial N$ are Legendrian isotopic to $L$. 
\end{lemma}

The following lemma was proven as Lemma~7.8 in \cite{EtnyreMinMukherjee22pre}.
\begin{lemma}\label{seestab}
  Let $(T^2\times[0,1], \xi)$ be a $\pm$ basic slice with dividing slopes $s_i$ on $T^2\times\{i\}$. Let $L_0$ be a Legendrian ruling curve of slope $s_1$ on $T^2\times\{0\}$ and $L_1$ a Legendrian divide on $T^2\times \{1\}$. Then $L_0$ is $S_\pm(L_1)$. Moreover, if $L'_0$ is a Legendrian divide on $T^2\times \{0\}$ and $L'_1$ is a ruling curve of slope $s_0$ on $T^2\times\{1\}$, then $L'_1$ is $S_\mp(L'_0)$.

  Let $s$ be a vertex in the Farey graph outside the interval $[s_0,s_1]$ for which there are vertices in the Farey graph in $[s,s_0)$ with an edge to $s_1$.  If $L''_i$ is a ruling curve of slope $s$ on $T^2\times\{i\}$ then $L''_0$ is $S_\pm^k(L''_1)$ where $k=|(s_1 \ominus s_0) \bigcdot s|$. Moreover, there is a similar statement when $s$ is outside of $[s_0,s_1]$ for which there are vertices in the Farey graph in $(s_1,s]$ with an edge to $s_0$, and with the roles of $L''_0$ and $L''_1$ interchanged as in the previous paragraph. 
\end{lemma}

We now discuss how to compute the classical invariants of a Legendrian realization of a cable.   Suppose $L_{p,q}$ is a Legendrian realization of the $(p,q)$-cable of $K$ and is contained in $\partial N({K})$, which is convex. Then as was shown in \cite{etnyre2005cabling}, see also  \cite[Theorem~1.7]{EtnyreLaFountainTosun12}, $\tb(L_{(p,q)})$ and $\rot(L_{(p,q))}$ can be computed as follows.
   \begin{lemma}
   	\label{lemma:tb}
	The Thurston-Bennequin invariant can be computed as follows:
   \begin{enumerate}
   \item Suppose $L_{(p,q)}$ is a Legendrian divide and slope $(\Gamma_{\partial{N(\mathcal{K})}})=\frac{q}{p}$. Then $\tb(L_{(p,q})=pq$.
   \item Suppose $L_{(p,q)}$ is a Legendrian ruling curve and slope $(\Gamma_{\partial{N(\mathcal{K})}})=\frac{q'}{p'}$. Then $\tb(L_{(p,q})=pq-|pq'-p'q|$.
   
   \end{enumerate}
   \end{lemma}
   
   \begin{lemma}
   	\label{lemma:rot}
   Let $D$ be a convex meridional disk of $N(K)$ with Legendrian boundary on a contact isotopic copy of the convex surface $\partial N(\mathcal{K})$, and let $\Sigma (L)$ be a convex Seifert surface with Legendrian boundary $L\in \mathcal{L}(\mathcal{K})$ which is contained in a contact isotopic copy of $\partial N(\mathcal{K})$. Then 
   
   \[r(L_{(p,q)})=q\cdot r(\partial D)+p\cdot r(\partial\Sigma(K))\]
   \end{lemma}

We end this section by discussing neighborhoods of Legendrian knots in non-standard coordinates as discussed in Section~\ref{conttori}. Suppose $S_r$ is a torus with lower meridian $r$ in a contact manifold $(M,\xi)$. If $\xi$ is tight when restricted to $S_r$ and $S_r$ has convex boundary with two dividing curves of slope $s$ such that $s$ and $r$ are connected by an edge in the Farey graph, then $S_r$ is a standard neighborhood of a Legendrian knot $L$ and the contact framing is $s$. If we stabilize $L$ then the resulting knot will have a standard neighborhood with lower meridian $r$ and dividing curves of slope $s'$, where $s'$ is clockwise of $r$, anti-clockwise of $s$, and has an edge in the Farey graph to both $r$ and $s$.

\subsection{Bypasses}
Given a convex surface $\Sigma$ in a contact manifold $(M,\xi)$ a \dfn{bypass} for $\Sigma$ is an embedded disk $D$ such that $\partial D=\alpha_1\cup \alpha_2$ is Legendrian, $\xi$ is tangent to $D$ at $\alpha_1\cap \alpha_2$, $D\cap \Sigma=\alpha_1$, $\alpha_1$ intersects the dividing curves on $\Sigma$ three times $\{x_1,x_2,x_3\}$, with $x_2$ on the interior of $\alpha_1$ and $\partial \alpha_1=\{x_1,x_3\}$, and $\xi$ is not tangent to $D$ along the interior of $\alpha_2$. We can assume that $\xi$ is tangent to $D$ at the $x_i$'s and the sign of the singularity at $x_2$ is the sign of the bypass. Let $\Sigma\times [0,1]$ be a small $[0,1]$-invariant neighborhood of $\Sigma$ so that $\Sigma=\Sigma\times \{0\}$ and part of $D$ intersects the neighborhood. We can push this neighborhood past $D$, still denoted $\Sigma \times [0,1]$, so that $\Sigma\times \{1\}$ is convex and the dividing curves differ from $\Sigma$'s in a prescribed way, see \cite{Honda00a}. We say $\Sigma\times \{1\}$ is obtained from $\Sigma$ by \dfn{attaching the bypass $D$}.  The details of this will not concern us here, but we indicate how attaching bypasses to convex tori changes their dividing curves and how to find bypasses. 

If $T$ is a convex torus and $D$ is a bypass for $T$, then attaching the bypass will affect the dividing cures of $T$ in one of the following ways:
\begin{enumerate}
\item increase the number of dividing curves by two,
\item decrease the number of dividing curves by two, or
\item change the slope of the dividing curves (this can only happen if $T$ had only two dividing curves). 
\end{enumerate}
We elaborate on the last possibility. First recall if $T$ is an oriented surface then there is a positive side and negative side to $T$. The positive side is the side such that if a transverse vector pointing in this direction followed by an oriented basis for $T$ gives an oriented basis for $M$. 
\begin{lemma}
Suppose $T$ is a convex torus with two dividing curves of slope $s$ and ruling curves of slope $r$. If $D$ is a bypass on the positive side of $T$ and we attach $D$ to $T$ to get the convex torus $T'$. Then $T'$ will have two dividing curves of slope $s'$ where $s'$ clockwise of $s$, anti-clockwise of $r$ and is the point in the Farey graph closest to $r$ with an edge to $s$. The region between $T$ and $T'$ is a basic slice and the sign of the basic slice agrees with the sign of the bypass. 

If $D$ is on the negative side of $T$ and we attach $D$ to $T$ to get the convex torus $T'$. Then $T'$ will have two dividing curves of slope $s'$ where $s'$ anti-clockwise of $s$, clockwise of $r$ and is the point in the Farey graph closest to $r$ with an edge to $s$. The region between $T$ and $T'$ is a basic slice and the sign of the basic slice is the opposite of the sign of the bypass. 
\end{lemma}

A common way to find bypasses is the imbalance principle. Suppose $A$ is a convex annulus with boundary Legendrian and one component in the convex surface $\Sigma$ and the other in the convex surface $\Sigma'$. If $A$ intersects the dividing set $\Gamma_\Sigma$ more times that it intersects the dividing set $\Gamma_{\Sigma'}$, then one may isotopy $A$ so that there is a bypass for $\Sigma$ on $A$.

\subsection{Rationally null-homologous knots}\label{rational}
The main reference for the material in this section is \cite{BakerEtnyre12} and the reader is refereed there for details. 
Recall a knot $K$ in a manifold $M$ is called \dfn{rationally null-homologous} if it is trivial in $H_1(M,\Q)$. This means that there will be some minimal integer $r$ such that $rK$ is trivial in $H_1(M,\Z)$. We call $r$ the order of $K$. Let $N(K)$ be a tubular neighborhood of $K$ and $M_K=M\setminus N(K)$ be the exterior of the neighborhood. It is not hard to see that, given some framing on $\partial N(K)$, there is some $(r,s)$-curve on $\partial N(K)$ that bounds an embedded surface $\Sigma'$ in $M_K$. One may extend $\Sigma'$ in $N(K)$ by ``coning" $\partial \Sigma'$. This gives a surface $\Sigma$ in $M$ whose interior is embedded but its boundary wraps $r$ times around $K$. We call $\Sigma$ a \dfn{rational Seifert surface} for $K$. 
We note that $\Sigma$ might not have connected boundary if $r$ and $s$ are not relatively prime, but we will not consider this case as all our examples below will have $\partial \Sigma$ connected and we assume that throughout the rest of our discussion here. We assume that $K$ is oriented and this will induce an orientation on $\Sigma$. 

If $K'$ is another oriented knot in $M$ that is disjoint from $K$ then we define the \dfn{rational linking number} to be
\[
\lk_\Q(K,K')=\frac 1r \Sigma\cdot K'
\]
where $\Sigma\cdot K'$ is the algebraic intersection between $\Sigma$ and $K'$. There is some ambiguity in this definition, see Section~1 of \cite{BakerEtnyre12}, but this will not be an issue for the situations considered here. 

Now suppose that $K$ is a Legendrian knot in $(M,\xi)$. As discussed in Section~\ref{kot} we know that $K$ has a standard neighborhood $N(K)$ with convex boundary having two dividing curves determined by the contact framing. Let $K'$ be one of the Legendrian divides on $\partial N(K)$. We define the \dfn{rational Thurston-Bennequin invariant} of $K$ to be 
\[
\tb_\Q(K)=\lk_\Q(K,K').
\]
To define the rational rotation number we first note that we have an immersion $i:\Sigma\to M$ that is an embedding on the interior of $\Sigma$ and an $r$-to-$1$ mapping on $\partial \Sigma$. (Notice that we are slightly abusing notation thinking of $\Sigma$ as an abstract surface and also as its image under $i$, we hope the meaning to be clear from context.) We can now consider $i^*\xi$ as an oriented $\R^2$-bundle over $\Sigma$. Since $\Sigma$ is a surface with boundary, we know $i^*\xi$ can be trivialized: $i^*\xi=\Sigma\times \R^2$. Let $v$ be a non-zero vector field tangent to $\partial \Sigma$ inducing the orientation on $K$. Using the trivialization of $i^*\xi$ we can think of $v$ as a map from $\partial \Sigma$ to $\R^2$. We can now define the \dfn{rational rotation} of $K$ to be
\[
\rot_\Q(K)=\frac 1r \text{winding}(v,\R^2).
\]
Notice that $\text{winding}(v,\R^2)$ is equivalent to the obstruction to extending $v$ to a non-zero vector field over $i^*\xi$ and thus can be interpreted as a relative Euler number. 

\section{The existence of non-loose knots}

In this section, we prove our results concerning when a knot type admits a non-loose Legendrian and transverse representative and the number of such representatives. 

\subsection{Determining which knots admit non-loose representatives} We begin by proving Theorem~\ref{characterize} that gives a necessary and sufficient condition for a knot type $K$ in a $3$-manifold $M$ admitting a non-loose Legendrian representative in some overtwisted contact structure. To do so, we first need the following lemmas.

\begin{lemma}\label{lem:s1s2}
  The knot $S^1\times \{pt\}$ in $S^1 \times S^2$ does not admit non-loose Legendrian and transverse representatives in any overtwisted contact structure on $S^1 \times S^2$.
\end{lemma}

\begin{proof}
  Suppose that $L$ is a Legendrian representative of $S^1 \times \{pt\}$ in an overtwisted contact structure on $S^1 \times S^2$ and fix a framing of $L$. Then it has a standard neighborhood $N$ with two dividing curves of slope $n \in \mathbb{Z}$ with respect to the given framing of $L$. Then there is a unique tight contact structure on $N$ and $N^c = S^1\times S^2 \setminus N$ is a solid torus with two longitudinal dividing curves. If $L$ is non-loose, then the contact structure on $N^c$ should be tight and there is also a unique tight contact structure on $N^c$ with longitudinal dividing curves and hence $N\cup N^c$ is the unique tight structure on $S^1\times S^2$, which contradicts that we assumed the contact structure is overtwisted. Therefore, there are no non-loose Legendrian representatives of $S^1\times \{pt\}$ in $S^1\times S^2$. 

  Since a Legendrian approximation of a non-loose transverse knot is non-loose (see \cite[Proposition~1.2]{Etnyre13}, there are also no non-loose transverse representative of $S^1\times \{pt\}$ in $S^1\times S^2$.
\end{proof}

\begin{lemma}\label{lem:connectsum}
  Let $M$, $M'$ be oriented $3$-manifolds. If $K$ is a knot in $M$ which does not admit non-loose Legendrian ({\it resp.} transverse) representatives in any overtwisted contact structure on $M$, then $K$ also does not admit non-loose Legendrian ({\it resp.} transverse) representatives in any overtwisted contact structure on $M \# M'$. 
\end{lemma}

\begin{proof}
  Suppose $L$ is a non-loose Legendrian representative of $K$ in some overtwisted contact structure on $M \# M'$. Let $N(L)$ be a standard neighborhood of $L$. Then $M \# M' \setminus N(L)$ is tight. Colin \cite{Colin97} showed that the connect sum of two contact structures is tight if and only if each of the summand contact structures is tight. Thus $M \setminus N(L)$ and $M'$ are both tight. Since $M \# M'$ is overtwisted, $M$ must be overtwisted. This contradicts that there are no non-loose Legendrian representatives of $K$ in $M$. 

  The same argument works for non-loose transverse representatives. 
\end{proof}

Now we are ready to prove Theorem~\ref{characterize}.

\begin{proof}[Proof of Theorem~\ref{characterize}]
  We first consider non-loose Legendrian representatives. Assume that $K$ intersects an essential sphere once transversely. This implies that $M = S^1\times S^2 \# N$ and $K$ is the core of the $S^1 \times S^2$ summand. By Lemma~\ref{lem:s1s2}, the core of $S^1 \times S^2$ does not admit non-loose Legendrian representatives in any overtwisted contact structure, so $K$ also does not admit non-loose Legendrian representatives in any overtwisted contact structure on $M$ by Lemma~\ref{lem:connectsum}. Also, let $M = M' \# M''$ where $K \subset M''$ and $M'' \setminus K$ is irreducible. Assume $M'$ does not admit a tight contact structure. Suppose $L$ is a non-loose Legendrian representative of $K$ in $M$. Then $M \setminus N(L)$ is tight. Colin \cite{Colin97} showed that the connect sum of two contact structures is tight if and only if each of the summand contact structures is tight. Thus both $M'$ and $M'' \setminus N(L)$ should be tight. This contradicts the assumption that $M'$ does not admit a tight contact structure, so $K$ does not admit non-loose Legendrian representatives in $M$. 

  Now suppose $K$ does not intersect an essential sphere transversely once and $M'$ admits a tight contact structure. The first condition implies that $K$ is not the core of $S^1 \times S^2$ summand. We claim that $K$ admits a non-loose Legendrian representative in at least two overtwisted contact structures on $M''$ (for the unknot in $S^3$, at least one). There are two cases we need to consider. If the boundary of $M'' \setminus N(K)$ is compressible, then  $M''$ is a lens space (or $S^3$) and $K$ is a rational unknot. Corollary~\ref{courseRU} tells us that if $K$ is not the unknot in $S^3$ then it admits non-loose Legendrian representatives in at least two distinct overtwisted contact structures. For the unknot in $S^3$, it is known that there is a unique overtwisted contact structure on $S^3$ in which $K$ has non-loose representatives, see \cite{EliashbergFraser09}. 
  
  If the boundary of $M'' \setminus N(K)$ is incompressible, we can now apply the proof of \cite[Theorem~1.8]{Etnyre13} to construct two overtwisted contact structures on $M''$ in which $K$ admits infinitely many distinct non-loose transverse representatives. (There is a mistake in the proof of \cite[Theorem~1.8]{Etnyre13}. In \cite{Etnyre13}, it was claimed that for any $K$, which is not the unknot in $S^3$, $M\setminus N(K)$ is irreducible when $M$ is irreducible, but that is clearly not the case if $K$ is contained in a ball and $M$ is not $S^3$, but easily proven otherwise.) According to \cite[Proposition~1.2]{Etnyre13}, any Legendrian approximation of these transverse representatives is a non-loose Legendrian representative of $K$.  Now since $M'$ admits a tight contact structure, we apply Colin's result \cite{Colin97} again and $M \setminus N(L) = M' \# (M'' \setminus N(L))$ is still tight. Thus $K$ admits a non-loose Legendrian representatives at least two overtwisted contact structure on $M$ (when $K$ is the unknot in $S^3$, at least one).
 
  For non-loose transverse representatives, the same argument works except for the fact that rational unknots do not admit non-loose transverse representatives in any overtwisted contact structure on a lens space, see Corollary~\ref{courseRU}
\end{proof}

Corollary~\ref{characterize-irreducible} is immediate from Theorem~\ref{characterize}.

We now turn to the proof of Corollary~\ref{maincor} concerning when a knot in a Seifert fibered space admits a non-loose representative. 
\begin{proof}[Proof of Corollary~\ref{maincor}]
The only reducible Seifert fibered spaces are $S^1\times S^2$ and $\R P^3\# \R P^3$ (which is an $S^1$-bundle over $\R P^2$) \cite[Proposition~1.12]{Hatcher}. In \cite{LiscaStipsicz09}, Lisca and Stipsicz showed that the only Seifert fibered spaces that do not admit tight contact structures are the manifolds $M_n$ obtained by $2n-1$ surgery on the $(2,2n+1)$--torus knot. Thus the corollary follows form Corollary~\ref{characterize-irreducible} except for $S^1 \times S^2$ and $\R P^3 \# \R P^3$.

According to Theorem~\ref{characterize}, the knot $S^1\times \{pt\}$ is the only knot type in $S^1\times S^2$ that has no non-loose Legendrian representatives. 

Similarly, Combining the fact that $\R P^3$ admits a tight contact structure with Theorem~\ref{characterize} tells us that any knot type $K$ in $\R P^3\# \R P^3$ will always have non-loose Legendrian representatives. 
\end{proof}

\subsection{Finding infinitely many non-loose representatives of a knot}

 To prove Theorem~\ref{allhaveinfinity} we begin with a construction. Given a knot $K$ in an irreducible $3$--manifold $M$ that admits non-loose representatives (and is not a rational unknot), let $C$ be the complement of a tubular neighborhood of $K$. As discussed in the previous section, $C$ in irreducible and has incompressible boundary. From Theorem~6.1 of \cite{HondaKazezMatic00} $C$ admits a universally tight contact structure $\overline\xi$ with convex boundary having two meridional dividing curves. One can now glue $T^2\times [0,1]$ to $C$ and extend the contact structure to $\xi''_i$ so that $\xi''_i$ on $T^2\times [0,1]$ is a convex Giroux torsion $i$. As proven in Proposition~4.2 of \cite{HondaKazezMatic02} for the appropriate extension of $\overline{\xi}$ (recall there are two), $\xi''_i$ is universally tight. It was further shown in the proof of Proposition~4.6 of that paper, that the convex Giroux torsion of $\xi''_i$ is  one larger than the convex Giroux torsion of $\xi''_{i-1}$, and thus each $i$ gives a distinct contact structure $\xi''_i$. One may similarly add (the appropriate) half convex Giroux torsion to $\xi''_i$ to get a contact structure $\xi''_{i+\frac 12}$ and as above, for different $i$ all of these contact structures will be distinct. Notice that all these manifolds are canonically diffeomorphic to $C$, so we will consider all the contact structures as living on $C$. 

We can now glue a solid torus $S$ to $(C,\xi''_i)$ for any integer or half-integer $i$ to get an overtwisted contact structure as follows. We first extend $\xi''_i$ by adding a basic slice $(T^2\times[0,1], \xi')$ with dividing slopes $\infty$ and $n\in \Z$. There are two choices and as argued in Section~4 of \cite{HondaKazezMatic02} with one of these choice the resulting contact structure is tight, denote it $\xi'_{i,n}$. Finally glue a solid torus with its unique tight contact structure having convex boundary with dividing slope $n$. The resulting manifold is $M$ and the resulting contact structures will be denoted $\xi'_i$. 

We are now ready to prove Theorem~\ref{allhaveinfinity} saying that all knots that admit non-loose Legendrian representatives (other than a rational unknot) will admit infinitely many such representatives with a given framing in at least two overtwisted contact structures and when the knot is null-homologous we can refine the statement.
\begin{proof}[Proof of Theorem~\ref{allhaveinfinity}]
Notice that $\xi'_i$ is obtained from $\xi'_{i-1}$ by adding a Lutz twist. Thus all the contact structures $\{\xi'_i\}_{i\in\N}$ are the same, \cite{Geiges08}. Denote that contact structure $\xi$. Similarly all the contact structures $\{\xi'_i\}_{i\in \N+\frac 12}$ are the same and it will be denoted by $\xi'$. Notice that in $\xi'_i$ we can find a solid torus $S_n$ in $S$ with convex boundary having two dividing curves of slope $n$ (here we fix an arbitrary framing on $K$ to identify the framings on $K$ with integers). From Section~\ref{kot}, we know $S_n$ is the standard neighborhood of a Legendrian knot $L^{n,i}$. Moreover the complement of a standard neighborhood of $L^{n,i}$ is $\xi'_{i,n}$ on $C$. Thus for different $i$, the complements of the $L^{n,i}$ are different. Now notice that $\xi'$ is obtained from $\xi$ by a half Lutz twist on the transverse push-off of $L^{n,i}$. This is known to change the $d_3$ invariant of the contact structure by the self-linking number of the transverse knot \cite[Proof of Theorem~4.3.1]{Geiges08}. Since the $d_3$ invariant is always an integer or an integer modulo an even number, we see that this will change the parity of the $d_3$ invariant. Hence $\xi$ and $\xi'$ must be non-homotopic contact structures. This concludes the first part of Theorem~\ref{allhaveinfinity}.

We now turn to the second part of the theorem that refines the above result when $K$ is null-homologous. We start with the contact structure $\xi$. Notice that we could arrange for the basic slice attached to $C$ in the construction of $\xi'_i$ to be positive (if it was not just reverse the orientation on all the contact structures involved in the construction). Given this, notice that we can choose the solid tori $S_n$ above so that $S_n\subset S_{n+1}$ and $\overline{S_{n+1}- S_n}$ is a positive basic slice for all $n$. Now since $S_n$ is a standard neighborhood of $L^{n,i}$ we see that $S_-(L^{n,i})$ is $L^{i,n-1}$. Moreover, $S_+(L^{n,i})$ is loose since the complement of a standard neighborhood of $S_+(L^{n,i})$ is obtained from the complement of a standard neighborhood of $L^{n-1}$ by attaching a negative basic slice $B$. Thus the contact structure is overtwisted since we can find a positive basic slice $B'$ in it that shares a boundary with $B$ and has dividing slopes $n$ and $\infty$ (that is the contact structure on $B\cup B'$ is given by a path in the Farey graph that goes from $n-1$ to $n$ to $\infty$ and the first edge has a $+$ and the second has a $-$, so we can shorten the path with a mismatch of signs and so yields an overtwisted contact structure). 

From construction it is clear that $\tb(L^{n,i})=n$. To compute the rotation number recall that if $n<0$, then \cite{Kanda98} says it is simply given by $\chi(\Sigma_+)-\chi(\Sigma_-)$ where $\Sigma$ is a convex Seifert surface for $K$ with Legendrian boundary on the boundary of a standard neighborhood of $L^{n,i}$ and $\Sigma_\pm$ are the $\pm$ regions of the complement of the dividing curves. In the construction of the $\xi''_i$ given in \cite[Section~4]{HondaKazezMatic02} we can see that for $n>0$ there are an even number of boundary parallel circles in the dividing set and $n$ arcs that are boundary parallel. With our choices, there will be a subsurface of $\Sigma$ that is diffeomorphic to $\Sigma$ in $\Sigma_+$ and the $n$ arcs in the dividing set bounds disks in $\Sigma_-$. Thus we see that $\rot(L^{n,i})=\chi(\Sigma)-n$ when $n<0$, but as we know the $L^{n,i}$ for $n\geq 0$ negatively stabilize to $L^{n',i}$ for $n'<0$  we see the formula holds for all $n$.  If we relabel $L^{n,i}$ as $L_-^{n,i}$ then we have constructed the claimed knots in Theorem~\ref{allhaveinfinity}. 

Now if $-K$ is smoothly isotopic to $K$ then notice that if we set $L_+^{n,i}$ equal to $-L_-^{n,i}$, then this is the second claimed family of knots in Theorem~\ref{allhaveinfinity}.

An analogous argument works to given the Legendrian knots $L_\pm^{n,i+\frac 12}$ in $\xi'$. 
\end{proof}

\section{Rational unknots}\label{sec:ruk}
In this section we will first classify non-loose rational unknots in lens space and then in the following subsection prove that Figures~\ref{fig:unknotsLegendrian1} and~\ref{fig:unknotsLegendrian2} give surgery diagrams for all non-loose rational unknots. 

\subsection{Classification of non-loose unknots}
We begin by outlining how the classification of non-loose knots will work. First recall that $L(p,q)$ is a union of two solid tori $V_0$ and $V_1$. Fixing coordinates on $T^2$ we can think of $V_0$ as having lower meridian $-p/q$ and $V_1$ as having upper meridian $0$. If we fix two curves of slope $s$ on $\partial V_0=\partial V_1$, then a contact structure is determined by taking a tight contact structure on $V_0$ in $\Tight(S_{-p/q},s)$ and a tight contact structure on $V_1$ in $\Tight(S^0,s)$ and gluing them together (there is no guarantee that the result is tight). 

We will think of the core of $V_0$ as being the rational unknot $K_0$. If we are looking for non-loose realizations of $K_0$ then we can take any slope $s$ with an edge to $-p/q$ in the Farey graph. Now a tight contact structure on $V_0$ with convex boundary having two dividing curves of slope $s$ is a standard neighborhood of a Legendrian representative in the knot type of $K_0$. (See the end of Section~\ref{kot}.) So if it is non-loose the contact structure on $V_1$ must be tight. Notice that $s$ is either in $(-p/q, (-p/q)^c]$ or $[(-p/q)^a, -p/q)$ (recall our notation for intervals in the Farey graph from Section~\ref{fgraph}). If it is in the first interval then any contact structure on $V_1$ is in $Tight(S^0,s)$ and is given by a signed path going from $s$ clockwise to $0$ and glue in the unique contact structure on $V_0$  will give a tight contact structure on $L(p,q)$. So we ignore these as we only wish to consider non-loose knots in overtwisted contact structures. Now if $s$ is in $[(-p/q)^a, -p/q)$ then the contact structure on $V_0$ will have convex tori parallel to the boundary with dividing slope any number clockwise of $-p/q$ and anti-clockwise of $s$, in particular it will contain a convex torus with dividing slope $0$. Note when we glue the contact structure on $V_0$ to any tight contact structure on $V_1$ the result will be overtwisted since a Legendrian divide on the convex torus with dividing slope $0$ will bound a disk in $V_1$ and hence we have an overtwisted disk. Thus the non-loose Legendrian realizations of $K_0$ with ``framing" $s$ are in one to one correspondence with the elements of $\Tight(S^0,s)$. 

We now notice that the meridional disk for $V_1$ will provide a rational Seifert surface for $K_0$, which has order $p$ in $L(p,q)$. So the rational Thurston-Bennequin invariant is the rational linking of $K_0$ with a contact push-off of $K_0$, given by a Legendrian divide on $\partial V_0$ is $\frac 1p (0 \bigcdot s)$. Similarly if $e$ is the Euler class of the contact structure on $V_1$ then the rational rotation number is $\rot_\Q=\frac 1p e(D)$ where $D$ is the meridional disk for $V_1$. To see this we show that there is an annulus in $V_0$ extending $D$ to a rational Seifert surface for $K_0$ and the tangent vector to $K_0$ and to $\partial D$ can be extended to a nonzero section of the contact structure on the annulus. Thus the relative Euler class of the contact structure pulled back to the domain of the Seifert surface (see Section~\ref{rational}) is the same as the relative Euler class of $D$. This justifies the formula. To see the claim about the annulus, we note that there is a cover of $V_0$ such that $\partial D$ pulls back to (some number of copies of) a longitude. We now consider a model for this solid torus. We choose a framing so that the pull-back of $\partial D$ has slope $0$. Then the contact structure can be assumed to be $\ker(\sin 2\pi x \, dy + \cos 2\pi x\, dz)$ on $S^1\times \R^2$ thought of as $\R^3$ modulo $x\mapsto x+n$ for some positive integer $n$. Now consider the annulus $A=\{(x,y,z): z=0, y\geq 0\}$ in $S^1\times \R^2$. We can see that the annulus if foliated by ruling curves parallel to the $x$-axis showing that the vector field pointing in the positive $x$-direction orients the lift of $K_0$ and extends to the annulus. (We note here that it is essential that the curve $\partial D$ is positive with respect to the contact framing, meaning that it is not between the meridional slope and the contact framing, for the above model to hold, but that is the case here.) We can now take a quotient of a convex neighborhood of the $x$-axis to construct a model for $V_0$ with the claimed extension of $D$ to a rational Seifert surface for $K_0$. 

For the Euler class of a contact structure described above we note that it is the obstruction to extending a trivialization of $\xi$ over the $1$-skeleton to the $2$-skeleton. We can take $K_0$ to be the $1$-skeleton of the lens space and then (the extended as above) $D$ to be the $2$-skeleton. Since we can take a tangent vector to $K_0$ to be a trivialization of $\xi$ over the $1$-skeleton, we see that $e(\xi)$ evaluated on the generator of $H_2(L(p,q))$ coming from $D$ is simply $e(D)$ where $e$ is the relative Euler class of $V_0$ discussed above. 

The slopes in $[(-p/q)^a, -p/q)$ with an edge to $-p/q$ are simply $s_k= (-p/q)^a\oplus k(-p/q)$.
We now observe how stabilization works. Suppose that $L$ is a Legendrian knot corresponding to the unique tight contact structure on $V_0$ with slope $s_k$, for $k>0$, and some tight contact structure on $V_1$ in $\Tight(S^0,s_k)$. Now inside of $V_0$ there are two solid tori $S_\pm$ (smoothly isotopic to $V_0$) that has convex boundary having two dividing curves of slope $s_{k-1}$ such that the basic slice $V_0\setminus S_\pm$ is a $\pm$ basic slice $B_\pm$. Now $S_\pm$ is the standard neighborhood of $S_\pm(L)$ and the contact structure on the complement is given by the contact structure on $V_1$ union $B_\pm$. That is it is given by the path in the Farey graph describing the contact structure on $V_1$ extended by the edge describing $B_\pm$. This new path might not (usually wont) be a minimal path. If the path can be consistently shortened then the contact structure on the complement of $S_\pm$ is tight, and hence the Legendrian knot that $S_\pm$ is a neighborhood of is non-loose. If the path can be inconsistently shortened, then the complement of $S_\pm$ is overtwisted, and the Legendrian knot that $S_\pm$ is a neighborhood of is loose. 

We now consider $s_0=(-p/q)^a$. When we stabilize the knot corresponding to $V_0$ it will have neighborhood $S_\pm$ with dividing slope $(-p/q)^c$ and the basic slice $B_\pm=V_0\setminus S_\pm$ will contain convex tori parallel to the boundary with dividing slope any rational number clockwise of $(-p/q)^c$ and anti-clockwise of $(-p/q)^a$. In particular it will contain a convex surface with dividing slope $0$ and hence $V_1\cup B_\pm$ will be overtwisted. That is any stabilization of a Legendrian knot corresponding to $V_0$ with dividing slope $s_0$ will be loose. 

We now give the classification of Legendrian rational unknots in $L(p,1)$, recovering the results from \cite{GeigesOnaran15} and adding information about how the non-loose unknots are related by stabilization. 

\begin{proof}[Proof of Theorem~\ref{Lp1}]
Notice that $(-p)^a=-\infty$. So with the notation above there is a non-loose rational unknot $L$ with standard neighborhood $V_0$ with convex boundary having slope $\infty$. From above we see that $\tb_\Q(L)=\frac 1p$. Moreover the contact structure on $V_1$ is the unique tight contact structure on a solid torus with longitudinal dividing curves and the formula for the rotation number above gives $\rot_\Q(L)=0$. Our discussion above also shows that $S_\pm(L)$ is loose. 

Now consider $s_1=-p-1$. The contact structures on $V_1$ with this dividing slope correspond to signed paths from $-p-1$ clockwise to $-1$. That is there are $p$ edges in the path and they are all in a continued fractions block, so the contact structure on $V_1$ is determined by the number of $+$ signs. They can be between $0$ and $p$, that is there are $p+1$ non-loose knots, denote them by $L_k$ where $k$ is the number of $-$ signs in the path describing the contact structure on $V_1$. From the formulas above we see that $\tb_\Q(L_k)=(p+1)/p$ and $\rot_\Q(L_k)=1-2k/p$. We also know that the Euler class of the contact structure containing $L_k$ is $p-2k$. 

Notice that if we negatively stabilize $L_0$ then the contact structure on its compliment is described by the path from $-1$ anti-clockwise to $-p-1$ having all positive signs and a negative sign on the edge from $-p-1$ to $\infty$. Since we can inconsistently shorten this path, we see that it is overtwisted. Thus $S_-(L_0)$ is loose. Similarly, if we positively stabilize $L_0$ then we have the same path from $-1$ to $-p-1$ and then a $+$ on the edge from $-p-1$ to $\infty$. This path can be consistently shortened to the one that goes from $0$ to $\infty$, that is the contact structure on the complement of the standard neighborhood of $L$. That is $S_+(L_0)=L$. Similarly $S_-(L_p)=L$ and $S_+(L_p)$ is loose. All the other $L_k$ have complements described by paths with both $+$ and $-$ signs, so the same analysis as above shows that any stabilization of them will be loose (since we can always shuffle the basic slices in the continued fraction block so that the sign opposite to  the stabilizing sign appears adjacent to the basic slice corresponding to the stabilization). 

Now consider $s_l$ for $l>1$. The path describing the contact structures on $V_1$ consist of an edge from $s_l$ to $-p$ and then edges from $-p, -p+1, \ldots, -1$. Those last $p-1$ edges are all in a continued fraction block and the first edge is not. So there are $2p$ non-loose knots. We denote them by $L^l_{\pm,k}$ where $k$ is the number of $-$ signs in the continued fraction block and $\pm$ denotes the sign of the first edge in the path. The argument above shows that $L^l_{+,0}$ has positive stabilization $L^{l-1}_{+,0}$ if $l>2$ and $S_+(L^2_{+,0})=L_0$, while all the negative stabilizations are loose. Similarly for the $L^l_{-,p-1}$. Thus $L$ together with $L_0, L_p$ and all the $L^l_{+,0}$ and $L^k_{-,p-1}$ form a $\normalfont\textsf{V}$ based at $(0,1/p)$. 

We can similarly check that $L^l_{+, k}$ , for $l>2$ and $k=1,\ldots, p-2$, positively stabilizes to $L^{l-1}_{+,k}$ and becomes loose when stabilizing negatively, while $L^l_{-, k}$, for $l>2$ and $k=1,\ldots, p-2$, negatively stabilizes to $L^{l-1}_{-,k}$ and becomes loose when stabilizing positively. Now consider $L^2_{+,k}$ for $k=1, \ldots, p-1$ (since $L^2_{+,0}$ was discussed above). If we negatively stabilize the knot it becomes loose, but if we positively stabilize it we can consistently shorten the path describing the contact structure on the complement to see that we get $L_{k}$. Similarly, positively stabilizing $L^2_{-,k-1}$, for $k=0, \ldots, p-1$ (since $L^2_{-,p-1}$ was discussed above), gives a loose knot while negatively stabilizing the knot gives $L_{k}$. Thus we see that the $L_k$, for $k=1, \ldots, p-1$ are the vertex of a $\normalfont\textsf{V}$ based at $(1-2k/p, (p+1)/p)$. 
\end{proof}

We now turn to the classification of non-loose rational unknots in $L(p,p-1)$. 
\begin{proof}[Proof of Theorem~\ref{Lpp1}]
We first notice that $(-p/(p-1))^a=-(p-1)/(p-2)$. With the notation above We consider the non-loose unknot with standard neighborhood $V_0$ having lower meridional slope $-p/(p-1)$ and dividing slope $-(p-1)/(p-2)$. Its complement is the solid torus $V_1$ with upper meridian $0$ and dividing slope $-(p-1)/(p-2)$. There are two tight contact structures on such a torus. So there are two non-loose rational unknots $L_\pm$ with $\tb_\Q(L_\pm)=(p-1)/p$ and we can compute $\rot_Q(L_\pm)=\pm(1-2/p)$. As discussed above, we know any stabilization of either of these knots is loose. In addition, the Euler class of the contact structure containing $L_\pm$ is $\pm(2-p)$. Notice that the sign of the Euler class is opposite to the sign of the rotation number. This is because $-p/(p-1)$ is clockwise of $-(p-1)/(p-2)$. 

Now $s_1=-(2p-1)/(2p-3)$. The contact structures on $V_1$ with these boundary conditions are given by the signed paths with vertices $-(2p-1)/(2p-3)$, $-p/(p-1)$, $-1$, and  the two edges are part of a continued fraction block, so there are three possible contact structures and hence three non-loose knots: $L_\pm^1$ corresponding to all of the signs being $\pm$ and $K^1$ corresponding the the path with one $+$ and one $-$. The rational Thurston-Bennequin invariant of these Legendrian knots is $(2p-1)/p$ and $\rot_\Q(L^1_\pm)= \pm (2-1/p)$ while $\rot_\Q(K^1)=0$. As we argued in the proof of Theorem~\ref{Lp1} we have $S_\pm(L^1_\pm)=L_\pm$ and $S_\mp(L^1_\pm)$ is loose, while any stabilization of $K^1$ is loose. We also note that the Euler class of the contact structure containing $K^1$ is zero. 

For $s_k$, with $k>1$, we see the contact structures on $V_1$ are described by a path with two edges in the Farey graph and these edges are not in a continued fraction block. Thus there are four contact structures on $V_1$ and hence four non-loose knots. Denote by $L_\pm^k$ the Legendrian knot whose complement is given by all the signs being $\pm$ and denote by $K_\pm^k$ the Legendrian knot whose edge from $-1$ to $-p/(p-1)$ is $\pm$. Once again as in the proof of Theorem~\ref{Lp1} we see that $L_+$ together with the $L_+^k$ form a back slash based at $(-1+\frac 2p, \frac{p-1}p)$ while the $L_-$ and $L^k_-$ give a forward slash based at $(1-\frac 2p, \frac{p-1}p)$. Similarly the $K^1$ and $K^k_\pm$ give a $\normalfont\textsf{V}$ based at $(0, \frac{2p-1}p)$.
\end{proof}

We will now classify non-loose Legendrian rational unknots in $L(2n+1,2)$ were we must consider $\pm K_0$ and $\pm K_1$. 
\begin{proof}[Proof of Theorem~\ref{l2n12}]
We start with the non-loose representations of $K_0$. We have that $(-(2n+1)/2)^a= -n-1$ and so $V_0$ with dividing curves of this slope is a standard neighborhood of a non-loose Legendrian realization of $K_0$. The tight contact structure on the complement $V_1$ will be determined by a signed path in the Farey graph from $-n-1$ to $-1$. Notice that this path forms a continued fraction block so the contact structure is determined by the number of $-$ signs in the path. So we have $n+1$ non-loose Legendrian representatives $L_l$, for $l=0, \dots, n$. We easily compute $\tb_\Q(L_l)=(n+1)/(2n+1)$, $\rot_\Q(L_l)= (n-2l)/(2n+1)$, and the Euler class of the contact structure containing $L_l$ is $2l-n$. 

Now  $s_k=-(n(2k+1)+k+1)/(2k+1)$, here we are using the notation established at the start of this section so these are the slopes of non-loose Legendrian realizations of $K_0$. We analyzed $k=0$ above. For $k=1$, the path from $s_1$ to $-1$ has two jumps in a continued fraction block, that is from $s_1$ to $-(2n+1)/2$ and then to $-n$, and $(n-1)$ jumps from $-n$ to $-1$ and these are all in a continued fraction block. So there are $3n$ tight contact structures on $V_1$ with these boundary conditions. We denote by $L^1_{j,l}$ the non-loose Legendrian knot whose complement as $j$ of $-$ signs in the first continued fraction block, and $l$ in the second. So $j=0,1,2$ and $l=0,\ldots, n-1$. Notice that when $j=0$ then if we positively stabilize $L^1_{0,l}$ we will add a $+$ basic slice to $V_1$ with boundary slopes $s_1$ and $-n-1$. We can then consistently shorten the path to get $L_l$. However, if we negatively stabilize the knot then the contact structure will become overtwisted since there will be an inconsistent shortening of the path, i.e. $S_+(L^1_{0,l}) = L_l$ and $S_-(L^1_{0,l})$ is loose. Similarly $S_-(L^1_{2,l})=L_{l+1}$ and $S_+(L^1_{2,l})$ is loose. So $L^1_{j,l}$ will stabilize to one of the $L_l$ for $j=0$ or $2$. We also know that any stabilization of $L^1_{1,l}$ will be loose. It is easy to see that $\tb_\Q(L^1_{j,l})= (3n+2)/(2n+1)$, $\rot_\Q(L^1_{1,l})=(n-1-2l)/(2n+1)$ for $l=0,\ldots, n-1$, and the Euler class of the contact structures containing the $L^1_{1,l}$ is $n-1-2l$, for $l=0,\ldots, n-1$. 

Finally, for $k>1$ we see the path in the Farey graph from $s_k$ to $-1$ goes from $s_k$ to $-(2n+1)/2$ to $-n$, and then from $-n$ to $-1$. The first two jumps are not in a continued fraction block and the last  $n-1$ are. So there are $4n$ contact structures on $V_1$ with dividing curves of slope $s_k$ and corresponding non-loose rational unknots $L^k_{i, j,l}$ where $i=0,1$ is the number of $-$ signs on the edge from $s_k$ to $-(2n+1)/2$, $j=0,1$ is the number of $-$ signs on the jump from $-(2n+1)/2$ to $-n$ and $l$ is the number of $-$ signs from $-n$ to $-1$. Notice that the positive stabilization of $L^k_{0,j,k}$ is $L^{k-1}_{j-1,k}$ for $k>3$ and $L^1_{j,l}$ for $k=2$, while a negative stabilization of is loose. Similarly a negative stabilization of $L^k_{1,j,k}$ is $L^{k-1}_{j+1,l}$ for $k>3$ and $L^1_{j+1,l}$ for $k=2$, while a positive stabilization is loose. 

From the above we see that $L_l$ is the base for a back slash based at $(-n/(2n+1), (n+1)/(2n+1))$ if $l=0$, a forward slash based at $(n/(2n+1), (n+1)/(2n+1))$ if $l=n$, and a $\normalfont\textsf{V}$ based at $((n-2k)/(2n+1), (n+1)/(2n+1))$ for $l=1,\ldots, n-1$ (notice that the sign of the Euler class is opposite of the rotation number). Similarly $L^1_{1,l}$ is be base for a $\normalfont\textsf{V}$ based at $((n-1-2k)/(2n+1),(3n+2)/(2n+1))$ finishing the classification of non-loose realizations of $K_0$ (and $-K_0$). 

We now turn to $K_1$ (and $-K_1$).  Now $V_1$ will be the neighborhood of the Legendrian knot and since it has an upper meridional slope of $0$ the possible slopes on its dividing curves are $s_0=\infty$ and $s_k=1/k$ for $k\in \N$. The contact structures on the complement of $V_1$ with dividing slope $s_k$ are given by a signed path in the Farey graph from $-(2n+1)/2$ clockwise to $s_k$ (recall the first jump will not have a sign, so this is really a signed path from $-n$ to $s_k$). Thus for $k=0$ there will be exactly two contact structures since there is an edge from $-n$ to $\infty$, and thus two non-loose knots $L_\pm$. We compute that $\tb_\Q(L_\pm)=2/(2n+1)$ and $\rot_\Q(L_\pm)=\pm 1/(2n+1)$ and the Euler class of the contact structure containing $L_{\pm}$ is $\mp1$.  For $k=1$ the contact structures on $V_0$ are given by signed paths in the Farey graph from $-n$ clockwise to $1$. This path has $(n+1)$ jumps all in a continued fraction block, so there are $n+2$ contact structures determined by the number $l$ of $-$ on the path, and hence $n+2$ non-loose Legendrian knots $L_l$. As above we see that $L_0$ negatively stabilizes to $L_+$  and positively stabilizes to be loose, while $L_{n+1}$ positively stabilizes to $L_-$ and negatively stabilizes to be loose (notice that the sign of the Euler class is opposite to the sign of the rotation number). All the other $L_l$ become loose after any stabilization. We also compute that $\tb_\Q(L_l)= (2n+3)/(2n+1)$, $\rot_\Q=(2n+2-4l)/(2n+1)$ for $l=0,\ldots,n+1$, and the contact structure containing $L_l$ has Euler number $-(2n+2-4l)$. Finally, for $k>1$ the signed path determining the contact structure on $V_0$ has one jump from $1/k$ to $0$ and then a continued fraction block from $0$ to $-n$.  Thus there are $2(n+1)$ contact structures on $V_0$ and hence $2(n+1)$ non-loose Legendrian knots $L^k_{\pm, l}$. The same arguments as in the $K_0$ case tells us that $L_+$ the base for a forward slash based at $(1/(2n+1),  2/(2n+1))$, $L_-$ contributes a backward slash based at $(1/(2n+1), 2/(2n+1))$, and the $L_l$ with $l=1, \ldots, n$ is the base of a $\normalfont\textsf{V}$ based at $((2n-2(2l-1))/(2n+1), (2n+3)/(2n+1))$. 
\end{proof}

We can now give the general classification of non-loose rational unknots. 

\begin{proof}[Proof of Theorem~\ref{lensgeneral}]
Consider the lens space $L(p,q)$ where $q\not= 1$ or $p-1$ and let $-p/q=[a_0,a_1, \ldots, a_n]$. We first notice that $(-p/q)^a=[a_0,\ldots, a_{n-1}]$ and we denote this number by $-p''/q''$. Notice that $p'' = \overline{q}$ where $1 \leq \overline{q} \leq p$ and $q\overline{q} \equiv 1 \pmod p$. 

Notice that the dividing slopes of a standard neighborhood of non-loose representatives of $K_0$ are $s_k=-(p''+kp)/(q''+kq)$ for $k \geq 0$. To understand when we see a slash or a $\normalfont\textsf{V}$, we need to discuss the signed paths describing tight contact structures on $V_1$ with boundary conditions $s_k$. For $n \geq 2$, consider the path $P$ in the Farey graph from $r=[a_0,\ldots, a_{n-3}, a_{n-2}+1]$ clockwise to $-1$. Decorations on this path will correspond to $A=|(a_0+1)\cdots (a_{n-2}+1)|$ tight contact structures on $V_1$ with dividing slope $r$. We will see that for any $s_k$ the signed path describing a tight contact structure on $V_1$ will contain the path $P$. Specifically, for $s_0$ we see that the path from $-p''/q''$ to $-1$ consists of a continued fraction block of length $|a_{n-1}+1|$ from $-p''/q''$ to $r$ and then $P$. This gives $|a_{n-1}|A$ contact structures.  For $s_1$ we see that the path from $s_1$ to $-1$ consists of a continued fraction block of length $|a_n|$ from $s_1$ to $s = [a_0,\ldots, a_{n-1}+1]$, and then a continued fraction block of length $|a_{n-1}+2|$ from $s$ to $r$ followed by $P$. This gives $|a_n-1||a_{n-1}+1|A$ tight contact structures. For $k>1$ the path from $s_k$ to $-1$ consists of one jump from $s_k$ to $-p/q$, followed by a continued fraction block of length $|a_n+1|$ from $-q/p$ to $[a_0,\ldots, a_{n-1}+1]$, followed by a continued fraction block of length $|a_{n-1}+2|$ from $[a_0,\ldots, a_{n-1}+1]$ to $r$, and finally $P$. This gives $2|a_n||a_{n-1}+1|A$ contact structures. 

Fix a choice $P_s$ of signs on $P$. We will consider non-loose Legendrian knot $L^1_{i,j,P_s}$ with standard neighborhoods having boundary slope $s_1$ and whose complement is given by a path with $i$ negative signs and $|a_n|-i$ positive signs from $s_1$ to $[a_0,\ldots, a_{n-1}+1]$, and $j$ negative signs and $|a_{n-1}+1|-j$ positive signs from $[a_0,\ldots, a_{n-1}+1]$ to $r$, followed by $P_s$. Notice if we stabilize $L^1_{i,j,P_s}$ then its complement will be the complement of $L^1_{i,j,P_s}$ with a basic slice added to it. The basic slice will have slopes $s_0$ and $s_1$ and we will be able to shorten the resulting path. So if $i\not=0$, or $|a_n|$ the the resulting contact structure will not be tight, since we can shuffle a basic slice of the opposite sign to the stabilization to be adjacent to $s_1$ so when we shorten we have inconsistent signs. On the other hand if $i=0$ then a positive stabilization will be tight. To see which contact structure we get, we need to consider the non-loose Legendrian knot $L^0_{l,P_s}$ with standard neighborhood having boundary slope $s_0$ and whose complement is given by a path from $-p''/q''$ to $r$ with $l$ negative signs and $|a_{n-1}+1|-l$ positive signs followed by $P_s$. We easily see that $S_+(L^1_{0,j,P_s})$ is the same as $L^0_{j,P_s}$ while the negative stabilization is loose. Similarly $S_-(L^1_{|a_n|, j,P_s})$ is the same as $L^0_{j+1,P_s}$ while a positive stabilization is loose. 

Finally consider Legendrian knots $L^k_{\pm,i,j,P_s}$ with standard neighborhood having dividing slope $s_k$ and whose complement is given by a path with a $\pm$ basic slice from $s_k$ to $-p/q$, $i$ negative signs in the continued fraction block from $-p/q$ to $[a_0,\ldots, a_{n-1}+1]$, $j$ negative signs in the continued fractions block from $[a_0,\ldots, a_{n-1}+1]$ to $r$, and then $P_s$. As above we see that for $k>2$ we have $S_\mp(L^k_{\pm,i,j,P_s})$ is loose while $S_\pm(L^k_{\pm,i,j,P_s})$ is $L^{k-1}_{\pm,i,j,P_s}$. Similarly $S_\mp(L^2_{\pm,i,j,P_s})$ is loose while $S_+(L^2_{\pm, i,j,P_s})$ is $L^1_{i,j,P_s}$ while $S_-(L^2_{-,i,j,P_s})$ is $L^1_{i+1,j,P_s}$. 

Putting this together we see that the $L^0_{l,P_s}$ for $l=0$ is the base for a back slash, while $l=|a_{n-1}+1|$ is the base for a forward slash, there are $|(a_0+1)\cdots (a_{n-2}+1)|$ of each of these. All the other $L^0_{l,P_s}$ are the base of a $\normalfont\textsf{V}$, there are $\left|(a_0+1)\cdots (a_{n-2}+1)(a_{n-1}+2)\right|$ of these. We note that 
$\tb_\Q=p''/p$ for all of these knots and the rotation numbers of these non-loose knots and Euler classes of the contact structures in which they live can be computed as discussed in the beginning of this section. Notice that the sign of the Euler class is opposite to the sign of the rotation number since $-p/q$ is clockwise of $(-p/q)^a$.  
Finally the $L^1_{i,j,P_s}$ with $i\not=0$ or $|a_n|$ each give a $\normalfont\textsf{V}$, and there are  
\begin{align*}
  (|(a_n-1)(a_{n-1}+1)|-2|a_{n-1}+1|)|(a_0+1)\cdots (a_{n-2}+1)| = |(a_0+1)\cdots(a_n+1)|
\end{align*}
of these. These knots have $\tb_\Q=(p''+p)/p$ and the rotation numbers of these non-loose knots and Euler classes of the contact structures in which they live can be computed as discussed in the beginning of this section. Notice that the sign of the Euler class is opposite to the sign of the rotation number since $-p/q$ is clockwise of $(-p/q)^a$.  

We are left to see that non-loose realizations of $K_0$ occur in at least two contact structures. Suppose that $a_0\not=-2$. Notice that the path from $-p/q$ to $-1$ will go through $a_0+1$. Consider a paths $P_l$ in the Farey graph from $-p/q$ to $-1$ such that all the sign less than $a+1$ are $+$ and the jumps from $a_0+1$ to $-1$ have $l$ negative signs and $|a_0+2|-l$ positive signs. Then the Euler class of the contact structure on $L(p,q)$ given by this path will be $c+(|a_0+2| -2l)$ for $l=0,\ldots, |a_0+2|$, where $c$ is a constant determined by the signed path from $-p/q$ to $a_0+1$. Also notice that $p$ is at least $|a_0+1|$ (since it is between $a_0+1$ and $a_0$, it will be an iterated Farey sum of these numbers).  Thus the order of the second cohomology of $L(p,q)$ is at least $|a_0+1|$ so the Euler classes indicated above will provide at least two distinct values.

We now assume that $a_0=-2$ and $k$ is the smallest index such that $a_k\not=-2$ (note there will be such a $k$ since if all the $a_i$ are $-2$ then $-p/q$ would be $-n/(n-1)$ for some $n$ and we are assuming our lens space is not $L(n,n-1)$ which was dealt with earlier). In this case $-p/q$ will be between $-(k+1)/k$ and $-(2k+3)/(2k+1)$. To see this note that $-(k+1)/k$ is $[-2.\ldots, -2]$ where there are $k$ is the number of $-2$s, and  $-(k+2)/(k+1)$ has the same continued fraction expansion except $k+1$ is the number of $-2$s. From this we see that $[-2,\ldots, -2,-3]$, where the number of $-2$s is $k$, is $-(2k+3)/(2k+1)$. Thus $-(2k+3)/(2k+1)$ is clockwise of $-p/q$ and we must have that $-(k+1)/k$ is anti-clockwise of $-p/q$ (since dropping the last term in a continued fraction moves the number anti-clockwise). In particular, 
$-p/q$ is in the interval claimed. Notice that this implies that the second cohomology of $L(p,q)$ as order at least $2k+3$. 

The first continued fraction block in the path from $-p/q$ to $-1$, traversed anti-clockwise goes from $-1$ to $-(k+2)/(k+1)$ to $-(k+2)/(k+1)\oplus i(-(k+1)/k)$ for $i=1,\ldots, |a_k+1|$. Let $P_l$ be the signed path from $-p/q$ to $-1$ that is all plus signs except on the continued fraction block discussed above, and on that continued fraction block there are $l$ negative signs. One can easily check that each basic slice in the last continued fraction block contributes $\pm(k+1)$ to the Euler class of the contact structure. So the Euler class of the contact structure corresponding to $P_l$ is $c+(k+1)(|a_k+1|-2l)$, for $l=0, \ldots, |a_k+1|$, where $c$ is the contribution from the part of the path before the last continued fraction block. Once again, it is clear that there are at least two distinct contact structures supporting non-loose representatives of $K_0$. 

The same sort of analysis as above classifies non-loose representations of $K_1$ and the case $n=1$. 
\end{proof}
We now end by noting that any rational unknot, other than the unknot in $S^3$, has non-loose Legendrian representatives in at least two contact structures and has no non-loose transverse representatives. 
\begin{proof}[Proof of Corollary~\ref{courseRU}]
The first part follows from Theorems~\ref{Lp1}, \ref{Lpp1}, and~\ref{lensgeneral}. The fact that all non-loose Legendrian knots have a lower bound on their $\tb_\Q$ (as shown by the same set of theorems), we clearly have no non-loose transverse knots. 
\end{proof}

\subsection{Surgery diagrams for rational unknots}\label{sdru}

We will begin by considering the diagram on the left-hand side of Figure~\ref{fig:unknotsLegendrian1} and show that these surgery diagrams describe all non-loose Legendrian representatives of $K_1$ with $\tb_\Q=q/p$ in $L(p,q)$. To this end, we notice that the smooth Dehn surgery done on the black knot is $-1+(q/(p+q))=-p/(p+q)$ and performing a Rolfson twist yields surgery coefficient $-p/q$. So the diagram does indeed give the lens space $L(p,q)$. Moreover, one can see that the Rolfson twist changes coordinates on the torus by the matrix $\begin{bmatrix}1& -1\\ 0&1\end{bmatrix}$. This matrix also takes the path defining an overtwisted contact structure on $L(p,q)$ that starts at $0$ and goes anticlockwise to $-p/q$ and goes past $\infty$ to a path that starts at $0$ and goes anticlockwise to $-p/(p+q)$ and goes past $\infty$. Thus all the contact structures constructed in the previous section are also described by this surgery diagram. Finally, notice that the solid torus giving a neighborhood of $K_1$ in the construction in the previous section had upper meridian $0$ and dividing slope $\infty$. The change of coordinates above sends this to a solid torus with upper meridian $0$ and dividing slope $-1$. This solid torus is precisely a neighborhood of the gray curve in the figure, thus completing the justification that the given surgery diagram gives all $\tb_\Q=q/p$ representatives of $K_1$.

We now consider the diagram on the right-hand side of Figure~\ref{fig:unknotsLegendrian1} and show that these surgery diagrams describe the non-loose Legendrian representatives of $K_1$ with $\tb_\Q=(p+q)/p$ in $L(p,q)$. We first note that the smooth surgery diagram consists a Hopf link with components $U_1$ and $U_2$, and the smooth surgery coefficients on the components are $a_0$ and $s=[a_1,\ldots, a_n]$, respectively; thus the surgery diagram describes $L(p,q)$ and the contact structure is obviously overtwisted as we are doing $(+1)$-contact surgery on a stabilized knot. Moreover, the gray knot $K_1$ is clearly non-loose since performing Legendrian surgery on it will cancel the $(+1)$-contact surgery on $U_1$ and hence result in a tight manifold. 

We now wish to identify the precise non-loose knot realized by the surgery diagram. To this end, we begin by considering $L(-a_0,1)$. In this case, the diagram on the right of Figure~\ref{fig:unknotsLegendrian1} does not have the unknot with contact framing $(s+1)$. The diagram gives $-a_0+1$ distinct non-loose Legendrian representatives of $K_1$, which is exactly the number claimed in Theorem~\ref{Lp1}; that is, it is equal to the number of contact structures on $T^2\times I$ with dividing slopes between $a_0-1$ and $-1$. So these are all of the non-loose rational unknots in this case. Now $L(p,q)$ is obtained from $L(-a_0,1)$ by a further contact $(s+1)$ surgery on the surgery dual to the contact $(+1)$ surgered unknot in the diagram. The number of contact structures one can achieve through this is $|a_0-1||a_1+1|\cdots|a_n+1|$, the same as the number of contact structures on $T^2\times I$ with dividing slopes $a_0-1$ and $(-p/q)^a$. Also, this number coincides the number of tight contact structures on the complement of a standard neighborhood of $K_1$ with slope $s_1$ in the proof of Theorem~\ref{lensgeneral} (notice that the proof of Theorem~\ref{lensgeneral} only considers $K_0$. For $K_1$ we need to reverse the order of $a_0, \ldots, a_n$). Thus again, we see that all the contact structures, and non-loose knots, described in Theorem~\ref{lensgeneral} are realized in the surgery diagram.

Turning now to Figure~\ref{fig:unknotsLegendrian2}, we note that if we consider the smooth surgery diagram then the contact framing on $K_1$ is indeed the claimed framing (note the smooth surgery coefficient on the bottom-most black knot is $-1$ and the unknots directly above have coefficient $-2$, so one can do a sequence of blow-downs to see that $K_1$ has the claimed framing), and hence realize the desired non-loose representatives of $K_1$, {\em c.f.} \cite{EtnyreMinMukherjee22pre}.

\section{Cabling non-loose Legendrian knots}
In this section we prove our structure results concerning when non-looseness is preserved under cabling and find non-simple non-loose transverse knots that have no Giroux torsion in their complements. 
\subsection{Positive cables of non-loose knots}
In this section, we prove Theorem~\ref{thm:poscable} that says the standard $(p,q)$-cable $L_{p,q}$ of $L$ is non-loose if and only if $L$ is non-loose.

\begin{proof}[Proof of Theorem~\ref{thm:poscable}]
Given a Legendrian knot $L\in \mathcal{L}(K)$ and a rational number $q/p> \tb(L)$, let $L_{p,q}$ be the standard Legendrian $(p,q)$-cable of $L$.  Recall this means that we take a standard neighborhood $N'$ of $L$ so that its ruling curves have slope $q/p$ and $L_{p,q}$ is one of these ruling curves. It is clear that if $L$ is loose then so is $L_{p,q}$ since $L_{p,q}$ is contained in a standard neighborhood of $L$. We will now show that if $L$ is non-loose then so is $L_{p,q}$. We begin with some notation. 

Let $N$ be a standard neighborhood of $L_{p,q}.$ We can assume that the intersection of $\partial N'$ with the complement $\overline{(M-N)}$ in an essential annulus $A$ with $\partial A$ ruling curves on $\partial N$. We can make $A$ convex and take an $I$-invariant neighborhood $A\times[0,1]$ of $A$ so that $N\cup (A\times [0,1])$ (after rounding corners) is an $I$-invariant thickened torus with boundary components $T$ and $T'$. We can assume that $T$ is the "outer torus", meaning that $T$ bounds a solid torus $S$ that contains $N'$ (and $T'$). Notice $T$ is contact isotopic to $\partial N'$ and thus the dividing curves on $A$ run from one boundary component to the other.

We now assume that $L_{p,q}$ is loose and derive a contradiction. To this end, we know there exists an overtwisted disk $D$ disjoint from $N$. Note that, $D$ might intersect the annulus $A$. We may smoothly isotope $A$ to $A'$ fixing $L_{p,q}$ such that $A'\cap D=\emptyset$. Given the smooth isotopy from $A$ to $A'$, we apply Colin's discritization of isotopy \cite{Honda02} to find a family of annuli $A=A_0, A_1,\cdots A_n=A'$ where $A_{i+1}$ is obtained from $A_i$ by a single bypass attachment. We may take the union of $A_i$ with one of the annuli of $\partial N\setminus \partial A_i$ to obtain a torus $T_i$ that bounds a solid torus $S_i$ that contains the other annulus of $\partial N\setminus \partial A_i$ (that is $T_i$ is the ``outer torus" of the two tori formed from $A_i$ and one of the two annuli of $\partial N\setminus \partial A_i$). We say the bypass taking $A_i$ to $A_{i+1}$ is attached from the ``outside" if it is outside of $S_i$ and from the ``inside" otherwise. In the proof of the following lemma it will be important to note that because $T_i$ contains a portion of $\partial N$ that contains ruling curves, we can use Lemma~\ref{isruling} to assume that $L_{p,q}$ sits on $T_i$. We now have the following result concerning $S_i$.

\begin{lemma}\label{claim1}
For all $i$ we have that $S_i$ is tight and the standard neighborhood $N'$ of $L$ is always contained in $S_i$
\end{lemma}
We will prove this lemma below, but assuming it now, we complete the proof of Theorem~\ref{thm:poscable}. The lemma tell us that $S_n$ contains $N'$ and is tight. Notice that this implies the complement of $A_n$ in $\overline{(M-N)}$ is tight, contradicting the fact that the overtwisted disk is disjoint from $A_n$. Indeed to see this notice that the complement of $A_n$ in $\overline{(M-N)}$ consists of the complement of $S_n$ and a component contained in $S_n$. So our claim is true if the contact structure on $S_n$ and its complement are both tight. Now since $N'\subset S_n$ and the complement of $N'$ is tight so is the complement of $S_n$ and Lemma~\ref{claim1} says $S_n$ is tight. Thus the overtwisted disk $D$ cannot exist.  
\end{proof}

\begin{proof}[Proof of Lemma~\ref{claim1}]
We inductively assume that $S_{i-1}$ is tight and contains $N'$ (this is clearly true for $S_0$). We first notice that $S_i$ does not contain any convex tori parallel to its boundary of slope $\infty$. Indeed, since $S_{i-1}$ is tight, it has no such tori and thus if the bypass to get to $S_i$ was attached on the inside of $S_{i-1}$ then clearly $S_i\subset S_{i-1}$ has the claimed property. Now if the bypass is attached from the outside, then $S_{i-1}\subset S_i$ and we argue as follows. Suppose there is a convex torus $T$ with dividing slope $\infty$, notice that $S_i\setminus N'$ is contained in the complement of $N'$ and so is a tight thickened torus and $T$ separates it into two pieces $B$ and $B'$ with $B$ being a basic slice with one boundary component $T_0=\partial N'$ having dividing slope $\tb(L)$ and the other boundary component $T$ having dividing slope $\infty$. Notice $T$ is also a boundary component of $B'$, the other boundary component is $\partial S_i$. Moreover $S_{i-1}$ is contained in $N'\cup B$ and there is a convex torus $T'$ in $B$ with dividing slope $q/p$. Now let $L$ be a Legendrian divide on $T'$ and $L'$ be a Legendrian ruling curve on $\partial S_i$.  Suppose the sign of the basic slice $S_i\setminus S_{i-1}$ is positive (so the paths in the Farey graph from $T_0$ to $T'$ and from $T'$ to $\partial S_i$ are all positive). Then by Lemma~\ref{seestab} we see that $L'$ is obtained from $L$ by some number of positive stabilizations, while $L_{p,q}$ (the ruling curve on $T_0$) is obtained from $L$ by some number of negative stabilizations. Since $L_{p,q}$ sits on $\partial S_i$, Lemma~\ref{canstablize} implies that it is also a further stabilization of $L'$. However, this is not possible since we noted above that $L_{p,q}$ is obtained from $L$ by strictly negative stabilizations. Thus we have our claim that $S_i$ cannot contain a convex torus with dividing slope $\infty$.  

We now notice that the contact structure on $S_i$ is tight if it contains $N'$. To this end notice that $S_i\setminus N'$ is contained in the complement of $N'$ and so the contact structure here is tight. That is we have a tight contact structure on a thickened torus and we recover $S_i$ from it by gluing in the tight contact structure on $N'$ (recall $\partial N'$ has longitudinal dividing curves and so there is a unique contact structure on it). Gluing a tight contact structure on a thickened torus to a tight structure on a solid torus with longitudinal dividing curves will always produce a tight contact structure unless the thickened torus contains a convex torus with dividing slope $\infty$, but that is ruled out by our argument in the previous paragraph. So the contact structure on $S_i$ is tight if $N'\subset S_i$. 

Now $S_i$ is obtained by attaching a bypass to $\partial S_{i-1}$ from either the inside or outside of $S_{i-1}$. If the bypass is attached from the outside, then clearly $S_i$ contains $S_{i-1}$ and by induction it also contains $N'$ and we are done. So we are left to consider the case when the bypass is attached from the inside. In this case there are two cases to consider. If $r$ is the dividing slope of $\partial S_{i-1}$ then $r$ is either an integer or a non-integer rational number. 

If $S_{i-1}$ has dividing slope $r\in \Q\setminus \Z$ then there is a unique torus $S'$ with convex boundary with dividing slope $\lfloor{r}\rfloor$ and that torus in turn contains $N'$ by Corollary~\ref{lemma:rational}. Now $T_i$ is obtained from $T_{i-1}=\partial S_{i-1}$ by a bypass attached from inside $S_{i-1}$. Thus the dividing slope of $T_i$ will have an edge in the Farey graph to $r$. Thus it must be in the interval $[\lfloor{r}\rfloor, r]$. In particular it will contain a torus $T''$ with convex boundary with dividing slope $\lfloor{r}\rfloor$ which must be isotopic to $S'$ since there is a unique such torus in $S_{i-1}$. Thus $S_i$ will contain $N'$. 
	
We are now left to consider the case when the dividing slope of $\partial S_i$ is an integer $n$. By Corollary~\ref{subtori} we know the solid torus $N'$ in $S_{i-1}$ is determined by a signed path in the Farey graph associated to the contact structure on $S_{i-1}\setminus N'$. Now suppose the dividing slope on $S_i$ is also an integer (which must be $n-1$). If the signed path associated to $S_{i-1}\setminus N'$ contains both signs then no matter what sign is associated to $S_i\setminus S_{i-1}$ we know that $S_i$ will have to still contain $N'$ by Corollary~\ref{subtori}. So we now must consider the case that all the signs of the signed path associated to $S_{i-1}\setminus N'$ are the same, say positive, and opposite to that of $S_i\setminus N'$. We claim that cannot happen. To see this we consider three cases: (1) $q/p>n,$ (2) $q/p<n-1,$ and (3) $n-1<q/p<n$.  In Case~(1), let $L$ be a ruling curve of slope $q/p$   on $\partial S_{i-1}$ (by which we mean, isotop $\partial S_{i-1}$ through convex surfaces so that it has ruling curves of slope $q/p$) and $L'$ a ruling curve on $\partial S_i$ of slope $q/p$. Then Lemma~\ref{seestab} shows that $L'$ is obtained from $L$ by positive stabilizations, while $L_{p,q}$ is obtained from $L$ by negative stabilization, but Lemma~\ref{seestab} shows that since $L_{p,q}$ sits on $\partial S_i$ that it is obtained from $L'$ by some number of stabilizations (that could be positive or negative). This is a contradiction since $L_{p,q}$ cannot be obtained from $L$ by strictly negative stabilizations, but also by some positive and negative stabilizations. A similar argument works in Case~(2) and~(3). 
\end{proof}

\subsection{Negative cables of non-loose knots} 
In this subsection we prove Theorem~\ref{thm:negcable} that says 
the standard cable $L_{p,q}^\pm$ is non-loose for all $p/q\in(\tb(L)-1, \tb(L))$ if $S_\pm(L)$ is non-loose.
\begin{proof}[Proof of Theorem~\ref{thm:negcable}]
We begin by considering the case when $S_+(L)$ is non-loose (the case for $S_-(L)$ is similar). For $p/q\in (tb(L)-1, tb(L))$ let $L^+_{p,q}$ be the standard $(p,q)$-cable of $L$. Let $S$ be a standard neighborhood of $L$ and $S'$ a standard neighborhood of $S_+(L)$. By definition $L^+_{p,q}$ is a Legendrian divide on a convex torus in $S\setminus S'$.

Assume for contradiction that $L^+_{p,q}$ is loose. So the contact structure on the complement $X$ of a standard neighborhood of $L^+_{p,q}$ is overtwisted. Let $T_0=\partial S$. Notice that $T_0$ breaks $X$ into two pieces $C$, the complement of $S$, and $B$, containing $\partial X$, and the contact structure restricted to these pieces is tight (since $L$ is non-loose and $B$ is contained in $S$). We know that $T_0$ can be smoothly isotoped to be disjoint from an overtwisted disk and thus by Colin's discritization of isotopy \cite{Honda02} we know there is a sequence of convex tori $T_0,\ldots, T_k$ such that $T_{i+1}$ is obtained from $T_i$ by attaching a bypass and $T_k$ has an overtwisted disk in its complement. The $T_i$ cut $X$ into two pieces, one $C_i$ is diffeomorphic to $C$ and the other $B_i$ that is diffeomorphic to $B$.

We will now inductively show that the complement of $T_i$ is tight, contradicting the existence of an overtwisted disk. More precisely we will inductively show that either (1) the slope of the dividing curves on $T_i$ is greater than or equal to $\tb(L)$, in which case $B_i$ contains a convex torus $T_i'$ contact isotopic to $T_0$, or (2) the slope of the dividing curves on $T_i$ is less than or equal to $\tb(L)$, in which case the slope of the dividing curves is in $(q/p, \tb(L)]$ and $B_i$ contains a convex torus $T_i''$ contact isotopic to $\partial S'$.  

We first observe that our inductive hypothesis implies that the complement of $T_i$ is tight. Indeed, in Case~(1) we know that the complement of $T_0$, and hence $T_i'$, in $M$ is tight and $C_i$ is contained in this tight complement so it is tight. Let $S_i$ be the solid torus bounded by $T_i'$ in $M$ (and containing $\partial X$). The thickened torus $A_i$ cobounded by $T_i$ and $T_i'$ is tight (since it is in the complement of $T_i'$ which is contact isotopic to $T_0$) and so $B_i$ is obtained from $S_i$ (which has the unique tight structure on a solid torus with longitudinal dividing curves) by gluing $A_i$ and removing a neighborhood of $L^+_{p,q}$, which is of course tight. Thus the complement of $T_i$ is tight in this case. In Case~(2) let $S'_i$ be the solid torus bounded by $T''_i$. Since $T''_i$ is contact isotopic to $\partial S'$ we know $S'_i$ is a tight solid torus and its complement in $M$ is also tight. Since $C_i$ is contained in this complement, we know the contact structure on $C_i$ is tight. The thickened torus $A_i$ between $T''_i$ and $T_i$ in $M$ is tight since it is in the complement of $S''_i$ and as above we see that the solid torus bounded by $T_i$ it tight and $B_i$ is a subset of that. 

Suppose $T_i$ satisfies Case~(1) of the inductive hypothesis. if the bypass used to go from $T_i$ to $T_{i+1}$ is attached from the outside of $S_i$ then clearly $T_{i+1}$ also satisfies Case~(1) of the inductive hypothesis. If the bypass is attached from the inside, then an argument just like in the proof of Lemma~\ref{claim1} shows that $T_{i+1}$ also satisfies the inductive hypothesis if the slope $r_i$ of the dividing curves on $T_i$ is in $\Q\setminus \Z$. If $r_i=\tb(L)$, then we can consider that $T_i$ satisfies Case~(2) of the inductive hypothesis and proceed as in that case. Now assume that $r_i$ is an integer larger than $\tb(L)$. So the solid torus $S_i$ is a standard neighborhood of a Legendrian knot $L'$.  Again as in the proof of Lemma~\ref{claim1} we see that if the path in the Farey graph describing the contact structure on the thickened torus $A_i$ between $T_i$ and $T'_i$ contains both signs, then $T_{i+1}$ will satisfy the inductive hypothesis. So we assume that all the signs are the same, say $+$, and the bypass used to get $T_{i+1}$ from $T_i$ is a negative bypass. Then we know inside of $S_{i+1}$ there is a torus $T$ with dividing slope $r_i-1$ and the basic slice between $T_i$ and $T$ is negative. The solid torus $S''$ bounded by $T$ is a standard neighborhood of a Legendrian $L''$ and $L''=S_-(L')$. Moreover, we know $L$ is obtained from $L'$ by some number of positive stabilizations. Inside $S''$ we see a convex torus $T'$ with dividing slope $q/p$ and $L^+_{p,q}$ is a Legendrian divide on this torus. Thus between $T$ and $T'$ there is a convex torus $T''$ with two dividing curves of slope $\tb(L)$. The solid torus bounded by $T''$ is a neighborhood of a Legendrian knot $L''$ and since $L^+(p,q)$ is contained in this solid torus we see that it is also a standard cable of $L''$, but $L''$ is obtained from $L'$ by at least one negative stabilization. Thus the rotation number of $L''$ and $L$ are different. This will mean the rotation number of their standard $(p,q)$-cable will be different. This contradiction implies that the sign of the bypass going from $T_i$ to $T_{i+1}$ must agree with the sign of the basic slices in $A_i$, then again as in the proof of Lemma~\ref{claim1} we see that $T_{i+1}$ satisfies the inductive hypothesis. 

We now suppose that $T_i$ satisfies Case~(2) of the inductive hypothesis. If the slope $r_i$ of $T_i$ is $\tb(L)$ then notice $T_i$ must be contact isotopic to $T_0$  (since $T''_i$ is contact isotopic to $\partial S'$) and so if the bypass attached to $T_i$ to get $T_{i+1}$ is attached from the outside then we will clearly satisfy Case~(1) of the hypothesis. If the bypass is attached from the inside then we can proceed as below. 
If $r_i$ is not $\tb(L)$ then if the bypass is attached from the outside of $S_i$, then $T_{i+1}$ will clearly satisfy the inductive hypothesis since attaching the bypass will result in the dividing curves on $T_{i+1}$ having slope in $(r_i, \tb(L)]$. 

We now consider a bypass attached from the inside, and notice that the slope $r_{i+1}$ of the dividing curves on $T_{i+1}$ must be in $(q/p,r]$. To see this suppose it were not. Let $S_{i}$ be the solid torus $T_{i}$ bounds in $M$. If $r_{i+1}$ is less than $q/p$ then the torus $S_{i+1}$ contains a neighborhood of $L^+_{p,q}$ and hence in $S_{i+1}$ we can find a convex torus $T'$ containing $L^+_{p,q}$ and the slope of the dividing curves on $T'$ must be $q/p$, but this implies that the contact structure on $S_{i+1}$ is overtwisted and hence the one on $S_i$ is too, but we argued above that since $T_i$ satisfies the inductive hypothesis, $S_i$ is tight. Thus we know $r_{i+1}\in(p/q,r_i]$. 

Assume for now that $T_i$ has only two dividing curves. Since the contact structure on $S_i$ is tight, we know it is given by a signed minimal path in the Farey graph from $\tb(L)-1$ to $r_i$. Said another way, if $A_i$ is the thickened torus cobounded by $T_i$ and $T''_i$ in $M$, then that path describes the tight contact structure on $A_i$ and there is a unique contact structure on the solid torus bounded by $T''_i$. There is a unique convex torus (up to contactomorphism) in $A_i$ with dividing slope $q/p$ and two diving curves and $L^+_{p,q}$ is a neighborhood of one of the Legendrian divides. Now $T_{i+1}$ is a convex torus in $S_i$. If $T_{i+1}$ has more than two dividing curves then it is contained in an $I$-invariant neighborhood of $T_i$ and hence satisfies the inductive hypothesis. If $T_{i+1}$ has two dividing curves then it splits $A_i$ into two thickened tori and each has a contact structure described by paths in the Farey graph that when combined can be shortened to the one describing the contact structure on $A_i$. Thus inside of $S_{i+1}$ we can also find the unique convex torus $T$ with two dividing curves and slope $q/p$. This will be the same torus (up to contactomorphism) as the one in $S_i$ and hence in the solid torus bounded by  $T$, there is a torus contact isotopic to $\partial S'$ and hence $T_{i+1}$ satisfies the inductive hypothesis. 

Now if $T_i$ has more than two dividing curves then inside of $T_i$ is a torus $T$ with the same dividing slope and just two dividing curves and $T''_i$ is contained in the solid torus bounded by $T$. The slope of the dividing curves on $T_{i+1}$ will be the same as the slope of those on $T_i$ (though the number of dividing curves can change by two) and in $S_{i+1}$ there is a convex torus contact isotopic to $T$. Thus $T_{i+1}$ satisfies the inductive hypothesis. 
\end{proof}

We end this section by seeing how the standard positive and standard negative cables relate. 
\begin{proof}[Proof of Proposition~\ref{relation}]
This is a restatement of Lemma~\ref{seestab}.
\end{proof}

\subsection{Cables of transverse knots}
In this section we prove Theorem~\ref{transversecable} that shows the standard cable of a transverse knot is non-loose if and only if the original knot is non-loose.
\begin{proof}[Proof of Theorem~\ref{transversecable}]
Assume that $K$ is a non-loose transverse knot and let $L_n$ be its Legendrian approximations as discussed when we defined the standard cable of $K$ just before the statement of Theorem~\ref{transversecable}. We know that since $K$ is non-loose then all the $L_n$ are too, see \cite[Proposition~1.2]{Etnyre13}. Notice that by Theorem~\ref{thm:poscable} we know that $(L_n)_{p,q}$ are all non-loose. If $K_{p,q}$ is loose then its Legendrian approximations with sufficiently negative Thurston-Bennequin invariant are also loose, but for very negative $n$, $(L_n)_{p,q}$'s are Legendrian approximations with arbitrarily negative Thurston-Bennequin invariant that are all non-loose. Thus $K_{p,q}$ is non-loose.

If $K$ is loose then all of its Legendrian approximations $L_n$ with sufficiently negative $n$ are loose. Thus $(L_n)_{p,q}$ are all loose and so their transverse push-off is also loose, see \cite{Etnyre13}. 
\end{proof}

\subsection{Transversely non-loose and non-simple knots}
In this section we prove Theorem~\ref{thm:transnonsimple} that says in the contact manifold $(S^3,\xi_2)$ the $(2n+1,2)$-cable of the left handed trefoil has at least $n$ non-loose transverse representatives with $\self=2n+3$. 

Before giving the proof, we recall some results from \cite{EtnyreMinMukherjee22pre}. Let $C$ be the complement of a neighborhood of the left-handed trefoil in $S^3$. In Lemma~6.2 of \cite{EtnyreMinMukherjee22pre} it was shown there exist tight contact structures $\xi_n, n\in \N,$ on $C$ such that $\partial C$ is convex with two dividing curves of slope $1/n$ that have the property that any convex torus parallel to $\partial C$ has dividing slope $1/n$. For $n=1$ we can glue in a solid torus $N$ with a tight contact structure (which is unique since the dividing curves are longitudinal) to get a contact structure on $S^3$ and the solid torus is a standard neighborhood of the Legendrian knot $L$. Theorem~7.22 in \cite{EtnyreMinMukherjee22pre} shows that this contact structure is $\xi_2$.

Now for $n>1$ we can glue a solid torus to $(C,\xi_n)$, but now there are two choices for the solid torus $N_n^\pm$ (depending on the sign of the bypass on the thickened torus with slopes $0$ and $1/n$). According to the Lemma~6.2 of \cite{EtnyreMinMukherjee22pre} all these tori are ``non-thickenable" meaning that any solid torus containing them (and smoothly isotopic to them) has dividing slope $1/n$. Repeating the proof of Theorem~1.12 in \cite{EtnyreLaFountainTosun12} shows that if $S$ is any solid torus in $N_n^\pm$ that is smoothly isotopic to $N_n^\pm$ and has convex boundary with dividing slope in $(0,1/n)$ then $S$ cannot be thickened to a solid torus with boundary having dividing slope larger than $1/n$. Moreover, if $S$ is a solid torus in $N_n^\pm$ as above but with dividing slope equal to $0$ then it will thicken to $N$. Thus all the contact structures on $S^3$ constructed from the $\xi_n$ by gluing in $N_n^\pm$ are $\xi_2$, so all the non-thickenable tori are in one fixed contact structure. 

\begin{proof}[Proof of Theorem~\ref{thm:transnonsimple}]
First fix $n \in \mathbb{N}$. Then inside any $N_k^-$ in $(S^3,\xi_2)$ with $k\leq n$, there is a convex torus parallel to the boundary with two dividing curves of slope $2/(2n+1)$. Notice from above, this torus bounds a solid torus that cannot thicken past $N_k^-$. Let $L_k$ be a Legendrian divide on this torus. The contact framing of $L_k$ agrees with the torus framing so it has $\tb(L_k)=4n+2$. Using Lemma~\ref{lemma:rot} we see that $\rot(L_k)=2n-1$. In the proof of Theorem~1.7 in \cite{EtnyreLaFountainTosun12}, it is shown that the only convex torus on which $S_-^l(L_k)$ can sit is $\partial N_k^-$ and a simple state transition argument given there shows that all the $L_k$ are distinct as are all their negative stabilizations. Moreover, this shows that the transverse push-off of $L_k$, which we denote $K_k$, is a non-loose transverse knot and none of the $K_k$ are transversely isotopic. Also, we know $\self(K_k)=\tb(L_k)-\rot(L_k)=2n+3$. 

For any $k \in \mathbb{N}$, all $N_k^-$ are universally tight and the above state transition argument tells us that any torus that $L_k$ sits on should have a universally tight neighborhood. According to \cite[Theorem~1.10]{ChakrabortyEtnyreMin20Pre} (the theorem is about tight contact manifolds, but the similar argument works for overtwisted contact manifolds too, see also \cite{McCullough18}), The necessary and sufficient condition for $L_k$ to destabilize is that it sits on a mixed torus, which has a virtually overtwisted neighborhood. This implies that $L_k$ does not destabilize and the complement of a standard neighborhood of $L_k$ does not contain (half) Giroux torsion. Since $K_k$ is a push-off of $L_k$, the complement of $K_k$ also does not contain (half) Giroux torsion.
\end{proof}

\end{document}